\pdfoutput=1 

\documentclass[12pt]{alt2021} 

\usepackage{enumitem}
\usepackage{amssymb}
\usepackage{amsmath}

\usepackage{mathrsfs}		
\usepackage{dsfont}

\usepackage{url}					
\hypersetup{colorlinks=true, urlcolor=blue, linktoc = all, linktocpage=true}
\usepackage{booktabs}

\usepackage{hhline}
\usepackage{multirow}

\usepackage[utf8]{inputenc} 
\usepackage[T1]{fontenc}

\usepackage{multicol}
\usepackage{silence}
\usepackage{amsfonts,thmtools}
\usepackage{mathtools}
\usepackage{xparse}
\usepackage{enumitem} 
\usepackage{etoolbox}
\usepackage{mathtools}
\usepackage{complexity}
\usepackage{svg}
\usepackage{xcolor, soul}
\usepackage{tablefootnote}
\usepackage[bb=boondox]{mathalpha} 
\definecolor{lightgrey}{rgb}{0.9,0.9,0.9}
\sethlcolor{lightgrey}

\definecolor{mygray}{rgb}{0.6,0.6,0.6}

\usepackage[capitalise,nameinlink,noabbrev]{cleveref}

\newtheorem{theorem}{Theorem}
\newtheorem{lemma}[theorem]{Lemma}
\newtheorem{proposition}[theorem]{Proposition}
\newtheorem{remark}[theorem]{Remark}
\newtheorem{corollary}[theorem]{Corollary}
\newtheorem{definition}[theorem]{Definition}

\newtheorem{fact}[theorem]{Fact}

\newtheorem{assumption}[theorem]{Assumption}

\usepackage{tikz}			
\usepackage{times}

\title[Accelerated Riemannian Optimization: Handle Constraints \& Bound Geometric Penalties]{Accelerated Riemannian Optimization: Handling Constraints with a Prox to Bound Geometric Penalties}

\altauthor{%
 \Name{David Martínez-Rubio} \Email{\href{mailto:martinez-rubio@zib.de}{martinez-rubio@zib.de}}\\
 \addr Zuse Institute Berlin and Technische Universität\\ Berlin, Germany%
 \AND
 \Name{Sebastian Pokutta} \Email{\href{mailto:pokutta@zib.de}{pokutta@zib.de}}\\
 \addr Zuse Institute Berlin and Technische Universität\\ Berlin, Germany%
}
\hypersetup{pdfborder={0 0 0}}

\newif\ifTODO 
\TODOtrue

\usepackage[color=green]{todonotes}

\newcommand{\pa}[1]{\left( #1\right)} 
\newcommand{\norm}[1]{\| #1 \|} 
\newcommand{\abs}[1]{\lvert #1 \rvert}

\newcommand*\circledaux[1]{\tikz[baseline=(char.base)]{
    \node[shape=circle,draw,inner sep=0.8pt] (char) {#1};}}

\NewDocumentCommand{\circled}{m o }{%
    \IfNoValueTF{#2}{\circledaux{#1}}{\stackrel{\circledaux{#1}}{#2}}%
}

\newcommand{\defi}{\stackrel{\mathrm{\scriptscriptstyle def}}{=}}

\renewcommand*\R{\mathbb{R}}

\let\epsilon\varepsilon

\usepackage{pifont} 
\newcommand{\cmark}{\ding{51}}
\newcommand{\xmark}{\ding{55}}%
\newcommand{\yesmark}{{\color{green} \cmark}}
\newcommand{\nomark}{{\color{red} \xmark}}

\DeclareMathOperator*{\argmin}{arg\,min}

\makeatletter
\renewcommand\paragraph{\@startsection{paragraph}{4}{\z@}%
                                    {0ex \@plus0.5ex \@minus.2ex}%
                                    {-1em}%
                                    {\normalfont\normalsize\bfseries}}
\makeatother

\newcommand{\innp}[1]{\langle #1 \rangle}
\newcommand{\bigo}[1]{O( #1 )}
\newcommand{\bigol}[1]{O\left( #1 \right)}

\newcommand\blfootnote[1]{%
  \begingroup
  \renewcommand\thefootnote{}\footnote{#1}%
  \addtocounter{footnote}{-1}%
  \endgroup
}

\usepackage[natbib, backend=biber, maxcitenames=3, minalphanames=3, maxbibnames=99, style=alphabetic, hyperref, backref, useprefix=true, uniquename=false, doi=false,url=false,eprint=false]{biblatex}

\usepackage{csquotes}               

\bibliography{refs}

\DeclareCiteCommand{\cite}
  {\usebibmacro{prenote}}
  {\usebibmacro{citeindex}%
   \printtext[bibhyperref]{\usebibmacro{cite}}}
  {\multicitedelim}
  {\usebibmacro{postnote}}

\DeclareCiteCommand*{\cite}
  {\usebibmacro{prenote}}
  {\usebibmacro{citeindex}%
   \printtext[bibhyperref]{\usebibmacro{citeyear}}}
  {\multicitedelim}
  {\usebibmacro{postnote}}

\DeclareCiteCommand{\parencite}[\mkbibparens]
  {\usebibmacro{prenote}}
  {\usebibmacro{citeindex}%
    \printtext[bibhyperref]{\usebibmacro{cite}}}
  {\multicitedelim}
  {\usebibmacro{postnote}}

\DeclareCiteCommand*{\parencite}[\mkbibparens]
  {\usebibmacro{prenote}}
  {\usebibmacro{citeindex}%
    \printtext[bibhyperref]{\usebibmacro{citeyear}}}
  {\multicitedelim}
  {\usebibmacro{postnote}}

\DeclareCiteCommand{\citeauthor}
  {\usebibmacro{prenote}}
  {\ifciteindex
     {\indexnames{labelname}}
     {}%
   \printtext[bibhyperref]{\printnames{labelname}}}
  {\multicitedelim}
  {\usebibmacro{postnote}}

\DeclareCiteCommand{\footcite}[\mkbibfootnote]
  {\usebibmacro{prenote}}
  {\usebibmacro{citeindex}%
  \printtext[bibhyperref]{ \usebibmacro{cite}}}
  {\multicitedelim}
  {\usebibmacro{postnote}}

\DeclareCiteCommand{\footcitetext}[\mkbibfootnotetext]
  {\usebibmacro{prenote}}
  {\usebibmacro{citeindex}%
   \printtext[bibhyperref]{\usebibmacro{cite}}}
  {\multicitedelim}
  {\usebibmacro{postnote}}

\DeclareCiteCommand{\textcite}
  {\boolfalse{cbx:parens}}
  {\usebibmacro{citeindex}%
   \printtext[bibhyperref]{\usebibmacro{textcite}}}
  {\ifbool{cbx:parens}
     {\bibcloseparen\global\boolfalse{cbx:parens}}
     {}%
   \multicitedelim}
  {\usebibmacro{textcite:postnote}}

\newbibmacro{string+doiurlisbn}[1]{%
  \iffieldundef{doi}{%
    \iffieldundef{url}{%
      \iffieldundef{isbn}{%
        \iffieldundef{issn}{%
          #1%
        }{%
          \href{http://books.google.com/books?vid=ISSN\thefield{issn}}{#1}%
        }%
      }{%
        \href{http://books.google.com/books?vid=ISBN\thefield{isbn}}{#1}%
      }%
    }{%
      \href{\thefield{url}}{#1}%
    }%
  }{%
    \href{https://doi.org/\thefield{doi}}{#1}%
  }%
}

\DeclareFieldFormat{title}{\usebibmacro{string+doiurlisbn}{\mkbibemph{#1}}}
\DeclareFieldFormat[article,incollection,inproceedings]{title}%
    {\usebibmacro{string+doiurlisbn}{\mkbibquote{#1}}}

\usepackage{algorithm}

\usepackage{algcompatible}
\algnewcommand{\lst}{\texttt{lst}}
\algnewcommand{\slst}{\texttt{slst}}
\algnewcommand{\SEND}{\textbf{send}}

\newsavebox{\algleft}
\newsavebox{\algright}

\makeatletter
\newcounter{algorithmicH}
\let\oldalgorithmic\algorithmic
\renewcommand{\algorithmic}{%
  \stepcounter{algorithmicH}
  \oldalgorithmic}
\renewcommand{\theHALG@line}{ALG@line.\thealgorithmicH.\arabic{ALG@line}}
\makeatother



\input{definitions.tex}

\begin{document}

\maketitle

\blfootnote{\href{https://openreview.net/forum?id=UPZCt9perOn}{\normalcolor First circulated in May 2022.}}
\blfootnote{Most of the notations in this work have a link to their definitions. For example, if you click or tap on any instance of $\xastg$, you will jump to the place where it is defined as the global minimizer of the function we consider in this work.}

\begin{abstract}
    We propose a globally-accelerated, first-order method for the optimization of smooth and (strongly or not) geodesically-convex functions in a wide class of Hadamard manifolds. We achieve the same convergence rates as Nesterov's accelerated gradient descent, up to a multiplicative geometric penalty and log factors.  
    Crucially, we can enforce our method to stay within a compact set we define. Prior fully accelerated works \textit{resort to assuming} that the iterates of their algorithms stay in some pre-specified compact set, except for two previous methods of limited applicability. For our manifolds, this solves the open question in \citep{kim2022accelerated} about obtaining global general acceleration without iterates assumptively staying in the feasible set.

In our solution, we design an accelerated Riemannian inexact proximal point algorithm, which is a result that was unknown even with exact access to the proximal operator, and is of independent interest. For smooth functions, we show we can implement the prox step inexactly with first-order methods in Riemannian balls of certain diameter that is enough for global accelerated optimization.
\end{abstract}

\clearpage
\tableofcontents

\section{Introduction}\label{sec:introduction}

Riemannian optimization concerns the optimization of a function defined over a Riemannian manifold. It is motivated by constrained problems that can be naturally expressed on Riemannian manifolds allowing to exploit the geometric structure of the problem and effectively transforming it into an unconstrained one. Moreover, there are problems that are not convex in the Euclidean setting, but that when posed as problems over a manifold with the right metric, are convex when restricted to every geodesic, and this allows for fast optimization \citep{neto2006convex, bento2012subgradient, bento2015proximal, allen2018operator}. That is, they are geodesically convex (g-convex) problems, cf. \cref{def:g-convex_smooth}. Some applications of Riemannian optimization in machine learning robust covariance estimation in Gaussian distributions \citep{wiesel2012geodesic}, Gaussian mixture models \citep{hosseini2015matrix}, operator scaling \citep{allen2018operator}, computation of Brascamp-Lieb constants \citep{bennett2008brascamp}, Karcher mean \citep{zhang2016fast}, Wasserstein Barycenters \citep{weber2017frank}, include dictionary learning \citep{cherian2016riemannian,sun2016complete}, low-rank matrix completion \citep{DBLP:journals/siamsc/CambierA16,heidel2018riemannian,mishra2014r3mc,tan2014riemannian,vandereycken2013low}, optimization under orthogonality constraints \citep{edelman1998geometry,DBLP:conf/icml/CasadoM19}, and sparse principal component analysis \citep{genicot2015weakly,huang2019riemannian,jolliffe2003modified}. The first seven problems are defined  over Hadamard manifolds, which we consider in this work, and the first six are g-convex problems to which our results can be applied. In fact, the optimization in these cases is over symmetric spaces, which satisfy a property that one instance of our algorithm requires, cf. \cref{thm:instantiation}.

Riemannian optimization, whether under g-convexity or not, is an extensive and active area of research, for which one aspires to develop Riemannian optimization algorithms that share analogous properties to the more broadly studied Euclidean methods, such as the following kinds of Riemannian first-order methods: deterministic \citep{bento2017iteration,wei2016guarantees,zhang2016first}, adaptive \citep{kasai2019riemannian}, projection-free \citep{weber2017frank,weber2019nonconvex}, saddle-point-escaping \citep{criscitiello2019efficiently,sun2019escaping,zhou2019faster, criscitiello2020accelerated}, stochastic \citep{hosseini2019alternative,khuzani2017stochastic,tripuraneni2018averaging}, variance-reduced \citep{sato2017riemannian,kasai2018riemannian,zhang2016fast}, and min-max methods \citep{zhang2022minimax, jordan2022first}, among others. 

Riemannian generalizations to accelerated convex optimization are appealing due to their better convergence rates with respect to unaccelerated methods, specially in ill-conditioned problems. Acceleration in Euclidean convex optimization is a concept that has been broadly explored and has provided many different fast algorithms. A paradigmatic example is Nesterov's Accelerated Gradient Descent (\newtarget{def:acronym_accelerated_gradient_descent}{\AGD{}}), cf. \citep{nesterov1983method}, which is considered the first general accelerated method, where the conjugate gradients method can be seen as an accelerated predecessor in a more limited scope \citep{martinez2021acceleration}. There have been recent efforts to better understand this phenomenon in the Euclidean case \citep{allen2014linear,su2014differential,drori2014performance,wibisono2016variational, diakonikolas2017approximate,joulani2020simpler}, which have yielded some fruitful techniques for the general development of methods and analyses. These techniques have allowed for a considerable number of new results going beyond the standard oracle model, convexity, or beyond first-order, in a wide variety of settings \citep{tseng2008accelerated, beck2009fista,wang2015unified,DBLP:conf/stoc/ZhuO15,allen2016katyusha, allen2017natasha, carmon2017convex, diakonikolas2017accelerated, hinder2019near, DBLP:conf/colt/GasnikovDGVSU0W19, ivanova2021adaptive, kamzolov2020near, criado2021fast}, among many others. There have been some efforts to achieve acceleration for Riemannian algorithms as generalizations of \AGD{}, cf. \cref{sec:related_work}. These works try to answer the following fundamental question:

\begin{center}
\textit{Can a Riemannian first-order method enjoy the same rates of convergence as Euclidean \AGD{}?}
\end{center}

The question is posed under (possibly strongly) geodesic convexity and smoothness of the function to be optimized. And due to the lower bound in \citep{criscitiello2022negative}, we know the optimization must be under bounded geodesic curvature of the Riemannian manifold, and we might have to optimize over a bounded domain.

\paragraph{Main results} In this work, we study the question above in the case of finite-dimensional Hadamard manifolds $\M$ of bounded sectional curvature and provide an instance of our framework for a wide class of Hadamard manifolds. For a function $f:\M\to\R$ with a global minimizer at $\xastg$, let $\xInit \in \M$ be an initial point and $\RR$ be an upper bound on the distance $\dist(\xInit, \xastg)$. If $f$ is differentiable, $L$-smooth, and (possibly $\mu$-strongly) g-convex in a closed ball of center $\xastg$ and radius $\bigo{\RR}$, our algorithms obtain the same rates of convergence as \AGD{}, up to logarithmic factors and up to a geometric penalty factor, cf. \cref{thm:instantiation}. See \cref{table:comparisons:riemannian} for a succint comparison among accelerated algorithms and their rates. This algorithm is a consequence of the general framework we design:

\emph{Riemacon: A general accelerated Riemannian scheme.} We design a Riemannian accelerated inexact proximal point method that enjoys the same rates as the Euclidean accelerated proximal point method when approximating $\min_{x\in\X}\f(x)$, up to logarithmic factors and up to a geometric penalty factor, where $\f :\NN \subset \M  \to \R$ is a  g-convex (or strongly g-convex) function in $\X \subset\NN$, cf. \cref{thm:g_convex_acceleration}. Note $\f$ does not need to be smooth, provided access to the inexact prox step.

For differentiable and smooth functions, we show that with access to a (not necessarily accelerated) constrained linear subroutine for strongly g-convex and smooth problems, we can inexactly solve the proximal subproblem from a warm-start point to enough accuracy so it can be used in our accelerated outer loop, in the spirit of other Euclidean algorithms like Catalyst \citep{lin2017catalyst}. After building this machinery, we show that we are able to implement an inexact ball optimization oracle, cf. \citep{carmon2020acceleration}, as an instance of our solution. Crucially, the diameter $\D$ of this ball depends on $\RR$ and the geometry only, so in particular it is independent on the condition number of $\f$. We can use the linearly convergent algorithm in \citep{criscitiello2022negative} for the implementation of the prox subroutine and we show that iterating the application of the ball optimization oracle leads to global accelerated convergence. We also define a convex Euclidean projection oracle, and show that it allows to implement another subroutine that enjoys better geometric penalties.

Importantly, our algorithms obtain acceleration without an undesirable assumption that most previous works had to make: that the iterates of the algorithm stay inside of a \emph{pre-specified} compact set without any mechanism for enforcing or guarateeing this condition. This condition is not the same as assuming the iterates are bounded in \emph{some} compact set, see \hyperlink{sec:handling_constraints_to_bound_geometric_penalties}{this discussion}. All methods require some constraints to bound geometric penalties but to the best of our knowledge only two previous methods are able to enforce these constraints, and they apply to the limited settings of local optimization \citep{criscitiello2022negative} and constant sectional curvature manifolds \citep{martinez2021acceleration}, respectively. Techniques in the rest of papers resort to just assuming that the iterates of their algorithms are always feasible. Removing this condition in general, global, and fully accelerated methods was posed as an open question in \citep{kim2022accelerated}, that we solve for a wide class of Hadamard manifolds. The difficulty of constraining problems in order to bound geometric penalties as well as the necessity of achieving this goal in order to provide full optimization guarantees with bounded geometric penalties is something that has also been noted in other kinds of Riemannian algorithms, cf. \citep{hosseini2020recent}.

The question concerning whether there are Riemannian analogs to Nesterov's algorithm that enjoy similar rates is a question that, to the best of our knowledge, was first formulated in \citep{zhang2016first}. In particular, since Nesterov's \AGD{} uses a proximal operator of a function's linearization, they ask whether there is a Riemannian analog to this operation that could be used to obtain accelerated rates in the Riemannian case. We show that, instead, a proximal step with respect to the \textit{whole} function can be approximated efficiently in Hadamard manifolds and it can be used along with an accelerated outer loop.
Previously known Riemannian proximal methods either obtain asymptotic analyses, assume exact proximal computation, or work with approximate proximal operators by using different inexactness conditions as ours, and none of them show how to implement the proximal operators or obtain accelerated proximal point methods, cf. \cref{sec:related_work}. 

\begin{table}[h!]
    \centering
    \caption{Convergence rates of related works with provable guarantees for smooth problems over uniquely geodesic manifolds. Column \textbf{K?} $\rightarrow$ sectional curvature?, \textbf{G?} $\rightarrow$ global algorithm?: any initial distance to a minimizer is allowed. Here L and L$'$ mean they are local algorithms that require initial distance $\bigo{(\L /\mu)^{-3/4}}$ and $\bigo{(\L /\mu)^{-1/2}}$, respectively. Column \textbf{F?} $\rightarrow$ full acceleration?: dependence on $\L $, $\mu$, and $\epsilon$ like  AGD  up to possibly log factors. Column \textbf{C?} can enforce some constraints?: All methods require their iterates to be in some \textbf{pre-specified} compact set, but works with \nomark{} just assume the iterates will remain within the constraints. We use ${\protect\W} \defi \sqrt{\frac{\L }{\mu}}\log(\frac{\L \RR^2}{\epsilon})$. $^*$A mild condition on the covariant derivative of the metric tensor is required, cf. \cref{assump:bounded_curvature_tensor}, \cref{sec:proofs_ball_oracle}. $^{**}$With access to the convex projection oracle in \cref{sec:other_subroutine}. } 
    \label{table:comparisons:riemannian} 
\begin{tabular}{lllccccc} 
    \toprule
    \textbf{Method}   &  \textbf{g-convex} & \textbf{$\mu$-st. g-cvx} & \textbf{K?} & \textbf{G?} & \textbf{F?}& \textbf{C?}  \\
    \midrule
    \midrule
    \citep[\AGD{}]{nesterov2005smooth}                 & $O(\sqrt{\frac{\L\RR^2}{\epsilon}})$  & $\bigo{\W}$ & $0$ & \yesmark & \yesmark & \yesmark \\
    \midrule
    \citep{zhang2018towards} & -  & $\bigo{\W}$ &  {\footnotesize bounded} & L  & \yesmark & \nomark   \\ 
    \citep{ahn2020nesterov}                          & - & $\bigotilde{\frac{\L}{\mu} + \W}$ &  {\footnotesize bounded} & \yesmark & \nomark & \nomark  \\ 
    
    \citep{martinez2020global}  & $\bigotilde{\zetar^{\frac{3}{2}}\sqrt{\frac{\zetar}{\deltar}+\frac{\L\RR^2}{\deltar\epsilon}}}$ & $\bigotilde{\zetar^{\frac{3}{2}}\cdot\W}$ & {\footnotesize ctant.$\neq 0$} & \yesmark & \yesmark & \yesmark  \\
    \citep{criscitiello2022negative}  & -  & $\bigo{\W}$ &  {\footnotesize bounded$^*$} & L$'$ & \yesmark & \yesmark   \\ 
    \citep{kim2022accelerated}  & $\bigo{\zetar \sqrt{\frac{\L \RR^2}{\epsilon}}}$ & $\bigo{\zetar \cdot \W}$ & {\footnotesize bounded} & \yesmark & \yesmark & \nomark  \\
    
    \textbf{Theorem \ref{thm:instantiation}}  & ${\bigotilde{\zetar^2 \sqrt{\zetar+\frac{\L\RR^2}{\epsilon}}}}$ & $\bigotilde{\zetar^2\cdot \W}$ & {\footnotesize Hadamard$^*$  } & \yesmark & \yesmark & \yesmark \\
    \textbf{Appendix \ref{sec:other_subroutine}}$^{**}$  & ${\bigotilde{\zetar \sqrt{\zetar+\frac{\L\RR^2}{\epsilon}}}}$ & $\bigotilde{\zetar\cdot \W}$ & {\footnotesize Hadamard$^*$  } & \yesmark & \yesmark & \yesmark \\
    \bottomrule
\end{tabular}
\end{table}

\subsection{Preliminaries} \label{sec:preliminaries}

We provide definitions of Riemannian geometry concepts that we use in this work. The interested reader can refer to \citep{petersen2006riemannian, bacak2014convex} for an in-depth review of this topic, but for this work the following notions will be enough. A Riemannian manifold $(\MOnly,\mathfrak{g})$ is a real $C^{\infty}$ manifold $\MOnly$ equipped with a metric $\mathfrak{g}$, which is a smoothly varying, i.e., $C^{\infty}$, inner product. For $x \in \MOnly$, denote by $\newtarget{def:tangent_space}{\Tansp{x}}\MOnly$ the tangent space of $\MOnly$ at $x$. For vectors $v,w  \in \Tansp{x}\MOnly$, we denote the inner product of the metric by $\innp{v,w}_x$ and the norm it induces by $\norm{v}_x \defi \sqrt{\innp{v,v}_x}$. Most of the time, the point $x$ is known from context, in which case we write $\innp{v,w}$ or $\norm{v}$.  

A geodesic of length $\ell$ is a curve $\gamma : [0,\ell] \to \MOnly$ of unit speed that is locally distance minimizing. A uniquely geodesic space is a space such that for every two points there is one and only one geodesic that joins them. In such a case the exponential map $\newtarget{def:riemannian_exponential_map}{\expon{x}} : \Tansp{x}\MOnly\to \MOnly$ and the inverse exponential map $\exponinv{x}:\MOnly\to \Tansp{x}\MOnly$ are well defined for every pair of points, and are as follows. Given $x, y\in\MOnly$, $v\in \Tansp{x}\MOnly$, and a geodesic $\gamma$ of length $\norm{v}$ such that $\gamma(0) =x$, $\gamma(\norm{v})=y$, $\gamma'(0)=v/\norm{v}$, we have that $\expon{x}(v) = y$ and $\exponinv{x}(y) = v$. We denote by $\newtarget{def:distance}{\dist}(x,y)$ the distance between $x$ and $y$, and note that it takes the same value as $\norm{\exponinv{x}(y)}$. The manifold $\MOnly$ comes with a natural parallel transport of vectors between tangent spaces, that formally is defined from a way of identifying nearby tangent spaces, known as the Levi-Civita connection $\nabla$ \citep{levi1977absolute}. We use this parallel transport throughout this work. As all previous accelerated related works do, discussed in \cref{sec:related_work}, we assume that we can compute the exponential and inverse exponential maps, and parallel transport of vectors for our manifold.

Given a $2$-dimensional subspace $V \subseteq \Tansp{x}\MOnly$ of the tangent space of a point $x$, the sectional curvature at $x$ with respect to $V$ is defined as the Gauss curvature, for the surface $\expon{x}(V)$ at $x$. The Gauss curvature at a point $x$ can be defined as the product of the maximum and minimum curvatures of the curves resulting from intersecting the surface with planes that are normal to the surface at $x$. A Hadamard manifold is a complete simply connected Riemannian manifold whose sectional curvature is non-positive, like the hyperbolic space or the space of $n\times n$ symmetric positive definite matrices with the metric $\innp{X, Y}_{A} \defi \operatorname{Tr}(A^{-1}XA^{-1}Y)$ where $X, Y$ are in the tangent space of $A$. Hadamard manifolds are uniquely geodesic. Note that in a general manifold $\expon{x}(\cdot)$ might not be defined for each $v\in \Tansp{x}\MOnly$, but in a Hadamard manifold of dimension $n$, the exponential map at any point is a global diffeomorphism between $\Tansp{x}\MOnly\cong\R^n$ and the manifold, and so the exponential map is defined everywhere. We now proceed to define the main properties that will be assumed on our model for the function to be minimized and on the feasible set $\X$.

\begin{definition}[Geodesic Convexity and Smoothness] \label{def:g-convex_smooth}
    Let $\f:\NN\subset \MOnly \to \R$ be a differentiable function defined on an open set $\NN$ contained in a Riemannian manifold $\MOnly$. Given $\newtarget{def:riemannian_smoothness_of_F}{\L}\geq \newtarget{def:strong_g_convexity_of_F}{\mu} > 0$, we say that $\f$ is \emph{$\L$-smooth} in a set $\X \subseteq \NN$ if for any two points $x, y \in \X$, $\f$ satisfies
    \[
        \f(y) \leq \f(x) + \innp{\nabla \f(x), \exponinv{x}(y)} + \frac{\L}{2}\dist(x,y)^2. 
    \]
    
    Analogously, we say that $\f$ is \emph{$\mu$-strongly g-convex} in $\X$, if for any two points $x, y \in \X$, we have
    \[
    \f(y) \geq \f(x) + \innp{\nabla \f(x), \exponinv{x}(y)} + \frac{\mu}{2}\dist(x,y)^2.
    \] 
    If the previous inequality is satisfied with $\mu=0$, we say the function is \emph{g-convex} in $\X$. If $\f$ is not differentiable, we say $\f$ is $\mu$-strongly g-convex in $\X$ if for all $x, y \in \X$, and $t\in[0,1]$:
\[
    f(\expon{x}(t\cdot\exponinv{x}(x) + (1-t)\cdot\exponinv{x}(y))) \leq tf(x) + (1-t)f(y) - \frac{t(1-t)\mu}{2}d(x,y)^2.
\] 
Again, if the inequality is satisfied for $\mu=0$, we have g-convexity in $\X$. This definition coincides with the previous one when $\f$ is differentiable.
\end{definition}

We present the following fact about the squared-distance function, when one of the arguments is fixed. The constants $\zeta{\D}$, $\delta{\D}$ below appear everywhere in Riemannian first-order optimization methods because, among other things, \cref{fact:hessian_of_riemannian_squared_distance} yields Riemannian inequalities that are analogous to the equality in the Euclidean cosine law of a triangle, cf. \cref{lemma:cosine_law_riemannian}, and these inequalities have wide applicability in the analyses of Riemannian methods. 

\begin{fact}[Local information of the squared-distance]\label{fact:hessian_of_riemannian_squared_distance}
    Let $\MOnly$ be a Riemannian manifold of sectional curvature bounded by $[\newtarget{def:minimum_sectional_curvature}{\kmin}, \newtarget{def:maximum_sectional_curvature}{\kmax}]$ that contains a uniquely g-convex set $\X\subset \MOnly$ of diameter $\D<\infty$. Then, given $x,y\in\X$ we have the following for the function $\newtarget{def:squared_distance_function_div_by_2}{\Phi_x}:\M \to\R$, $y \mapsto \frac{1}{2}\dist(x, y)^2$:
\[
    \nabla \Phi_x(y) = -\exponinv{y}(x)\text{\quad \quad and \quad \quad} \delta{\D}\norm{v}^2 \leq \Hess  \Phi_x(y)[v, v] \leq \zeta{\D} \norm{v}^2.
\] 
These bounds are tight for spaces of constant sectional curvature. The geometric constants are
\begin{equation*}\label{eq:defi_zeta_and_delta}
    \newtarget{def:zeta}{\zeta{\D}} \defi 
        \begin{cases} 
            \D\sqrt{\abs{\kmin}}\coth(\D\sqrt{\abs{\kmin}})& \text{ if } \kmin \leq 0 \\
            1 &  \text{ if } \kmin > 0\\
        \end{cases}%
        ,
\end{equation*}
and 
\begin{equation*}\label{eq:defi_delta}
        \newtarget{def:delta}{\delta{\D}} \defi 
        \begin{cases} 
            1 & \text{ if }\kmax \leq 0\\
            \D\sqrt{\kmax}\cot(\D\sqrt{\kmax})& \text{ if } \kmax > 0 \\
        \end{cases}.
\end{equation*}
    Consequently, $\Phi_x$ is $\delta{\D}$-strongly g-convex and $\zeta{\D}$-smooth in $\X$. See \citep{kim2022accelerated}, for instance.  
In particular, for Hadamard manifolds, $\Phi_x$ is $1$-strongly g-convex and sublevel sets of g-convex functions are g-convex sets, so balls are g-convex in these manifolds \citep{bacak2014convex}.
\end{fact}

\subsection{Notation}\label{sec:notation} Let $\newtarget{def:M_manifold}{\M}$ be a uniquely geodesic $\newtarget{def:dimension}{\n}$-dimensional Riemannian manifold. Given points $x, y, z\in \M$, we abuse the notation and write $y$ in non-ambiguous and well-defined contexts in which we should write $\exponinv{x}(y)$. For example, for $v\in \Tansp{x}\M$ we have $\innp{v, y - x} = -\innp{v, x - y} = \innp{v, \exponinv{x}(y)- \exponinv{x}(x)} = \innp{v, \exponinv{x}(y)}$; $\norm{v-y} = \norm{v-\exponinv{x}(y)}$; $\norm{z-y}_x = \norm{\exponinv{x}(z)-\exponinv{x}(y)}$; and $\norm{y-x}_x =\norm{\exponinv{x}(y)} =\dist(y, x)$. We denote by $\newtarget{def:geodesically_convex_feasible_region}{\X}$ a compact, uniquely geodesic g-convex set of diameter $\newtarget{def:diameter_of_geodesically_convex_feasible_region}{\D}$ contained in an open set $\newtarget{def:N_open_set_in_manifold}{\NN} \subset\M$ and we use $\newtarget{def:indicator_function}{\indicator{\X}}$ for the indicator function of $\X$, which is $0$ at points in $\X$ and $+\infty$ otherwise. For a vector $v\in\Tansp{y}\M$, we use $\newtarget{def:parallel_transport}{\Gamma{y}{x}}(v) \in \Tansp{x}\M$ to denote the parallel transport of $v$ from $\Tansp{y}\M$ to $\Tansp{x}\M$ along the unique geodesic that connects $y$ to $x$. We call $\newtarget{def:riemannian_function_f}{\f}:\NN \subset \M\to\R$ a g-convex function we want to optimize. We use $\newtarget{def:accuracy_epsilon}{\epsilon}$ to denote the approximation accuracy parameter, $\xk[0] \in \X$ for the initial point of our algorithms, and $\newtarget{def:initial_distance_for_X}{\RInit} \defi \dist(\xk[0], \xast)$ for the initial distance to an arbitrary constrained minimizer $\newtarget{def:optimizer_at_X}{\xast} \in \argmin_{x\in\X}\f(x)$. We use $\newtarget{def:distance_or_bound_to_global_minimizer}{\RR}$ for an upper bound on the initial distance $\dist(\xInit,\xastg)$ to an unconstrained minimizer $\newtarget{def:global_optimizer}{\xastg}$, if it exists. The big-$O$ notation $\newtarget{def:big_o_tilde}{\bigotilde{\cdot}}$ omits $\log$ factors. 
Note that in the setting of Hadamard manifolds, the bounds on the sectional curvature are $\kmin \leq \kmax \leq 0$. Hence for notational convenience, we define $\newtarget{def:zeta_without_subindex_which_means_zeta_D}{\zetad} \defi \zeta{\D} = \D\sqrt{\abs{\kmin}}\coth(\D\sqrt{\abs{\kmin}}) \geq 1$, $\newtarget{def:delta_without_subindex_which_means_zeta_D}{\deltad} \defi 1$, and similarly $\newtarget{def:zeta_without_subindex_which_means_zeta_RR}{\zetar}\defi \zeta{\RR}$ and $\newtarget{def:delta_without_subindex_which_means_zeta_RR}{\deltar}\defi\delta{\RR}=1$.  If $v\in \Tansp{x}\M$, we use $\newtarget{def:euclidean_projection}{\Pii}{\newtarget{def:closed_ball}{\ball}(0, r)}(v) \in \Tansp{x}\M$ for the projection of $v$ onto the closed ball with center at $0$ and radius $r$.

\section{Algorithmic framework and convergence results}\label{sec:algorithm}
In this section, we present our \textbf{Riema}nnian \textbf{ac}celerated algorithm for \textbf{con}strained g-convex optimization, or \newtarget{def:Riemacon}{\riemacon{}}\footnote{Riemacon rhymes with ``rima con'' in Spanish.}. This is a general framework that we later instantiate to provide a full algorithm. Recall our abuse of notation for points $p \in\M$ to mean $\exponinv{q}(p)$ in contexts in which one should place a vector in $\Tansp{q}\M$ and note that in our algorithm $\xk$ and $\yk$ are points in $\M$ whereas $\zx \in \Tansp{\xk}\M, \zy, \zybar  \in \Tansp{\yk}\M$.

We start with an interpretation of our algorithm that helps understanding its high-level ideas. The following intends to be a qualitative explanation, and we refer to the pseudocode and the appendix for the exact descriptions and analysis. Euclidean accelerated algorithms can be interpreted, cf. \citep{allen2014linear}, as a combination of a gradient descent (\newtarget{def:acronym_gradient_descent}{\GD{}}) algorithm and an online learning algorithm with losses being the affine lower bounds $\f(\xk) + \innp{\nabla \f(\xk), \cdot-\xk}$ we obtain on $\f(\cdot)$ by applying convexity at some points $\xk$. That is, the latter builds a lower bound estimation on $\f$. By selecting the next query to the gradient oracle as a cleverly picked convex combination of the predictions given by these two algorithms, one can show that the instantaneous regret of the online learning algorithm can be compensated by the local progress \GD{} makes, up to a difference of potential functions, which leads to accelerated convergence. In Riemannian optimization, there are two main obstacles. Firstly, the first-order approximations of $\f$ at points $x_k$ yield functions that are affine but only with respect to their respective $\Tansp{x_k}\M$, and so combining these lower bounds that are only simple in their tangent spaces makes obtaining good global estimations not simple. Secondly, when one obtains such global estimations, then one naturally incurs an instantaneous regret that is worse by a factor than is usual in Euclidean acceleration. This factor is a geometric constant depending on the diameter $\D$ of a set $\X$ where the iterates and a (possibly constrained) minimizer lie. As a consequence, the learning rate of \GD{} would need to be multiplicatively increased by such a constant with respect to the one of the online learning algorithm in order for the regret to still be compensated with the local progress of \GD{} (and the rates worsen by this constant). But if we fix some $\X$ of finite diameter, because \GD{}'s learning rate is now larger, it is not clear how to keep the iterates in $\X$. And if we do not have the iterates in one such set $\X$, then our geometric penalties could grow arbitrarily. 

\begin{algorithm}[h!]
    \caption{Riemacon: \textbf{Riema}nnian \textbf{Ac}celeration - \textbf{Con}strained g-Convex Optimization}
    \label{alg:accelerated_gconvex}

\begin{algorithmic}[1] 
    \REQUIRE Feasible set $\X$. Initial point $\newtarget{def:initial_point}{\xInit} \in\X\subset\NN$. Function $\f:\NN\subset\M\to\R$ that is g-convex in $\X$, for a Hadamard manifold $\M$. Parameter $\lambda>0$. Optionally: final iteration $T$ or accuracy $\epsilon$. If $\epsilon$ is provided, compute the corresponding $T$, cf. \cref{thm:g_convex_acceleration}.

    \paragraph{Parameters:} 
    \begin{itemize}
        \item Geometric penalty $\newtarget{def:extra_geom_penalty_xi}{\xi} \defi 4\zeta{2\D} -3 \leq 8\zetad -3 =O(\zetad)$. 
        \item Implicit Gradient Descent learning rate $\newtarget{def:implicit_GD_learning_rate_lambda}{\lambda}$.
        \item Mirror Descent learning rates $\newtarget{def:MD_step_length}{\etak} \defi \ak/\xi$.
        \item Proportionality constant in the proximal subproblem accuracies: $\newtarget{def:loss_factor_in_Lyapunov}{\Deltak} \defi \frac{1}{(k+1)^2}$.
    \end{itemize}
    \paragraph{Definition:} (computation of this value is not needed) 
    \begin{itemize}
        \item Prox. accuracies: $\newtarget{def:accuracy_of_prox_subproblems}{\sigmak} \defi \frac{\Deltak \dist(\xk, \ykast)^2}{78\lambda}$ where $\newtarget{def:optimal_proximal_operator}{\ykast} \defi \argmin_{y\in\X}\{\f(y) + \frac{1}{2\lambda}\dist(\xk, y)^2 \}$.
    \end{itemize}

    \vspace{0.1cm}
    \hrule
    \vspace{0.1cm}
    \State $\yk[0] \gets \sigmak[0]$-minimizer of the proximal problem $\min_{y\in\X}\{\f(y) + \frac{1}{2\lambda}\dist(\xInit, y)^2\}$
    
    \State $\zx[0][0] \gets 0 \in \Tansp{\xInit}\M$;\quad $\zybar[0][0] \gets \zy[0][0] \gets 0 \in \Tansp{\yk[0]}\M$; \quad$\Ak[0] \gets 200\lambda\xi$ 
    \FOR {$k = 1 \textbf{ to } T$}
        \State $\newtarget{def:discrete_step}{\ak[k]} \gets 2\lambda\frac{k+32\xi}{5}$
        \State $\newtarget{def:integral_of_steps}{\Ak[k]} \gets \ak[k]/\xi + \Ak[k-1] = \sum_{i=1}^k \ak[i]/\xi + \Ak[0] = \lambda \left(\frac{k(k+1+64\xi)}{5\xi} + 200\xi\right)$
        \State $\newtarget{def:iterate_x}{\xk} \gets \expon{\yk[k-1]}(\frac{\ak}{\Ak[k-1] + \ak} \zybar[k-1][k-1] + \frac{\Ak[k-1]}{\Ak[k-1] + \ak}\yk[k-1])=  \expon{\yk[k-1]}(\frac{\ak}{\Ak[k-1] + \ak}\zybar[k-1][k-1]) $\Comment{Coupling}  
        
        \State $\newtarget{def:iterate_y}{\yk} \gets$ $\sigmak$-minimizer of $y\mapsto\{\f(y) + \frac{1}{2\lambda}\dist(\xk, y)^2\}$ in $\X$ \label{line:subroutine}\Comment{Approximate implicit \RGD{}}
        \State $\newtarget{def:v_k_sup_x}{\vkx[k]} \gets -\exponinv{\xk}(\yk)/\lambda$  \Comment{Approximate subgradient}
        \State $\newtarget{def:iterate_z_x}{\zx[k]} \gets \exponinv{\xk}(\expon{\yk[k-1]}(\zybar[k-1][k-1])) - \etak \vkx[k]$ \Comment{Mirror Descent step} \label{line:MD_step}
        \State $\newtarget{def:iterate_z_y}{\zy} \gets \Gamma{\xk}{\yk}(\zx) + \exponinv{\yk}(\xk)$ \Comment{Moving the dual point to $\Tansp{\yk}\M$}
        \State $\newtarget{def:iterate_z_y_bar}{\zybar} \gets \Pii{\bar{B}(0, \D)}(\zy) \in \Tansp{\yk}\M$ \Comment{Easy projection done so the dual point is not very far} \label{line:dual_projecting_step}
    \ENDFOR
    \State \textbf{return} $\yk[T]$.
\end{algorithmic}
\end{algorithm}

We find the answer in implicit methods. An implicit Euclidean (sub)gradient descent step is one that computes, from a point $\xk \in \X$, another point $\ykast = \xk - \lambda v_k \in \X$, where $v_k \in \partial(\f + \indicator{\X})(\ykast)$, is a subgradient of $\f+\indicator{\X}$ at $\ykast$. Intuitively, if we could implement a Riemannian version of an implicit \GD{} step then it should be possible to still compensate the regret of the other algorithm and keep all the iterates in the set $\X$. Computing such an implicit step is computationally hard in general, but we show that approximating the proximal objective $\newtarget{def:proximal_objective}{\hk}(y) \defi \f(y) + \frac{1}{2\lambda} \dist(\xk, y)^2$ with enough accuracy yields an approximate subgradient that can be used to obtain an accelerated algorithm as well.
In particular, we provide an accelerated scheme for which we show that the error incurred by the approximation of the subgradient can be bounded by some terms we can control, cf. \cref{lemma:the_suitable_approximate_prox_has_the_property_needed_in_alg1}, namely a small term that appears in our Lyapunov function and also a term proportional to the squared norm of the approximated subgradient, which only increases the final convergence rates by a constant. For $\L$-smooth functions, we provide a warm start in \cref{lemma:warm_start_riemannian_criterion_2} and show that an unaccelerated linearly convergent subroutine initialized at the warm-started point achieves the desired accuracy of the subproblem fast, cf. \cref{remark:RGD_can_be_used_as_subroutine}.
This proximal approach works by exploiting the fact that the Riemannian Moreau envelop is g-convex in Hadamard manifolds \citep{azagra2005inf} and that the subproblem $\hk$, defined with $\lambda=\zeta{2\D}/\L $, is strongly g-convex and smooth with a condition number that only depends on the geometry. For this reason, a local algorithm like the one in \citep{criscitiello2022negative} can be implemented in balls whose radius is independent on the condition number of $\f$. Besides these steps, we use a coupling of the approximate implicit \RGD{} and of a mirror descent (\newtarget{def:acronym_mirror_descent}{\MD{}}) algorithm, along with a technique in \citep{kim2022accelerated} to move dual points to the right tangent spaces without incurring extra geometric penalties, that we adapt to work with dual projections, cf. \cref{lemma:compensated_geometric_penalty}. Importantly, the \MD{} algorithm keeps the dual point close to the set $\X$ by using the projection in Line \ref{line:dual_projecting_step}, which implies that the point $\xk$ is close to $\X$ as well, and this is crucial to keep low geometric penalties. This \MD{} approach is a mix between follow-the-regularized-leader algorithms, that do not project the dual variable, and pure mirror descent algorithms that always project the dual variable. In the analysis, we note that partial projection also works, meaning that defining a new dual point that is closer to all of the points in the feasible set but without being a full projection leads to the same guarantees. Because we use the mirror descent lemma over $\Tansp{\yk}\M$, what we described translates to: we can project the dual $\zy$ onto a ball defined on $\Tansp{\yk}\M$ that contains the pulled-back set $\exponinv{\yk}(\X)$ and by means of that trick we can keep the iterates $\xk$ close to $\X$. And at the same time, the point for which we prove guarantees, namely $\yk$, is always in $\X$.

Finally, under $\L$-smoothness, we instantiate our subroutine with the algorithm in \citep{criscitiello2022negative}, in balls of radius independent on the condition number of $\f$ and show in \cref{thm:instantiation} that if we iterate this approximate implementation of a ball optimization oracle, we obtain convergence at a globally accelerated rate. In \cref{sec:other_subroutine} we provide a warm start that allows any Riemannian constrained linearly convergent algorithm to serve as subroutine for our algorithm and we provide another subroutine that allows to reduce our geometric penalties, provided that one can implement the convex projection operator defined in this section. We note \citep[Thm. 15]{zhang2016first} provided a claimed linearly convergent algorithm for constrained strongly g-convex smooth problems, and thus in principle it could be used for our subroutine after the warm start. Unfortunately, we noticed that the proof is flawed when the optimization is constrained. The first inequality in their proof only holds in general for unconstrained problems and not for projected Riemannian gradient descent, not even for the Euclidean constrained case. Thus, to the best of our knowledge there is no convergence analysis for this metric-projected \RGD{} in this setting. In \cref{sec:RGD_and_proj_RGD} we provide an analysis for this algorithm when the diameter of the feasible set is smaller than a constant ($\zetad<2$) and the global minimizer is inside of the set. We show in \cref{prop:metric_projection_in_balls} that metric projections onto Riemannian balls are simple to compute for uniquely geodesic sets, which implies that both our warm start and projected \RGD{} in the aforementioned setting admit a simple implementation. Finally, we include in \cref{thm:rgd} a result on the convergence of unconstrained \RGD{} with curvature independent rates.

We leave the proofs of most of our results to the appendix and state our main theorems below. Using the insights explained above, we show the following inequality on $\psik$, defined below, that will be used as a Lyapunov function to prove the convergence rates of \cref{alg:accelerated_gconvex}.

\begin{proposition}\label{thm:psi_is_lyapunov}\linktoproof{thm:psi_is_lyapunov}
    By using the notation of \cref{alg:accelerated_gconvex}, let
\[
    \newtarget{def:function_for_Lyapunov_analysis}{\psik} \defi \Ak(\f(\yk)-\f(\xast)) + \frac{1}{2} \norm{\zy-\exponinv{\yk}(\xast)}_{\yk}^2 + \frac{\xi-1}{2} \norm{\zy}_{\yk}^2.
\] 
    Then, for all $k\geq 1$, we have $(1-\Deltak)\psik \leq \psik[k-1]$.
\end{proposition}

With this proposition, we can show the convergence of \riemacon{} for g-convex functions.

\begin{theorem}\label{thm:g_convex_acceleration}\linktoproof{thm:g_convex_acceleration}
    Let $\M$ be a finite-dimensional Hadamard manifold of bounded sectional curvature, and consider $f:\NN\subset\M\to\R$ be a g-convex function in a compact g-convex set $\X \subset\NN$ of diameter $\D$, $\lambda \in \R_{>0}$, $\xast \in \argmin_{x\in\X} \f(x)$, and $\RInit \defi \dist(\xInit, \xast)$. For any $\epsilon > 0$, \cref{alg:accelerated_gconvex} yields an $\epsilon$-minimizer $\yk[T] \in \X$ after $T=\bigo{\zetad\sqrt{\frac{\RInit[2]}{\lambda\epsilon}}}$ iterations. If the function is $\mu$-strongly g-convex then, via a sequence of restarts, we converge in $\bigo{(\zetad\sqrt{\frac{1}{\lambda\mu}}+1)\log(\frac{\mu\RInit[2]}{\epsilon})}$ iterations.
\end{theorem}

We note that a straightforward corollary from our results is that if we can compute the exact Riemannian proximal point operator and we use it as the implicit gradient descent step in Line \ref{line:subroutine} of \cref{alg:accelerated_gconvex}, then the method is an accelerated proximal point method. One such Riemannian algorithm was unknown in the literature as well. Note we do not require smoothness of $\f$. 

Finally, we instantiate \cref{alg:accelerated_gconvex} to implement approximate ball optimization oracles in an accelerated way. We show that applying these oracles sequentially leads to global accelerated convergence. Moreover, we show that the iterates do not get farther than $2\RR$ from $\xastg$, which ultimately leads to the geometric penalty being a function of $\zetar$ and not on the condition number of $\f$. For the subroutine in Line \ref{line:subroutine} of \cref{alg:accelerated_gconvex}, we use the algorithm in \citep[Section 6]{criscitiello2022negative}, and for that we require the following.

\begin{assumption}\label{assump:bounded_curvature_tensor}
    Let $\newtarget{def:curvature_tensor}{\curvtensor}$ be the curvature tensor of a Riemannian manifold $\M$. Its covariant derivative is $\nabla \curvtensor = 0$.
\end{assumption}

Locally symmetric manifolds, like the SPD matrix manifold, manifolds of constant sectional curvature, $\operatorname{SO}(n)$, the Grasmannian manifold, are all manifolds such that $\nabla \curvtensor = 0$. We argue that this assumption is mild, since in particular these manifolds cover all of the applications in \cref{sec:introduction}.

\begin{algorithm}
    \caption{Boosted Riemacon: ball optimization boosting of a \riemacon{} instance (\cref{alg:accelerated_gconvex})}
    \label{alg:instance_of_riemacon}

\begin{algorithmic}[1] 
    \REQUIRE Differentiable function $\f:\NN\subset\M\to\R$ that is $\L $-smooth and $\mu$-strongly g-convex in $\ball(\xastg, 3\RR)\subset \NN$; initial point $\xInit \in \NN$; bound $\RR \geq \dist(\xInit, \xastg)$; accuracy $\epsilon$.

       \riemaconsc{}: The strongly convex version of \cref{alg:accelerated_gconvex} in \cref{thm:g_convex_acceleration} (cf. {\hyperlink{proof:thm:g_convex_acceleration}{its proof}}).
    \vspace{0.1cm}
    \hrule
    \vspace{0.1cm}
    
    \State \textbf{if} $2\RR \leq (46\RR \abs{\kmin} \zeta{2\RR})^{-1}$ \textbf{then return} $\text{\riemaconsc}(\ball(\xInit, \RR), \xInit, \f, \zeta{4R}/\L, \epsilon)$ \label{line:if_D_equal_R_one_single_Riemacon} \label{line:first_if_in_instanced_alg}
    \State Compute $\D$ such that $\D = (46\RR \abs{\kmin} \zeta{\D})^{-1}$. Alternatively, make $\D\gets (70 \RR\abs{\kmin})^{-1}$.\label{line:computing_D_in_instanced_alg}
    
    \State $\newtarget{def:total_number_of_iterations_of_instanced_algorithm}{\TT} \gets \lceil \frac{4\RR}{\D} \ln(\frac{\L\RR^2}{\epsilon}) \rceil$; \ \  $\newtarget{def:accuracy_epsilon_prime_in_ball_subproblem}{\epsilonp} \gets \min\{\frac{\D\epsilon}{8\RR},\frac{\mu\RR^2}{2\TT^2}\}$; \ \ $\oldlambda \gets \zeta{2\D}/\L$
    \FOR {$k = 1 \textbf{ to } \TT$}
       \State $\X_k \gets \ball(x_{k-1}, \D/2)$
       \State $x_{k} \gets \text{\riemaconsc}(\X_k, x_{k-1}, \f, \oldlambda, \epsilonp)$ \Comment{\citep{criscitiello2022negative} as subroutine} \label{line:riemacon_w_subroutine}
    \ENDFOR
    \State \textbf{return} $x_{\TT}$.
\end{algorithmic}
\end{algorithm}

\begin{theorem}\label{thm:instantiation}\linktoproof{thm:instantiation}
    Let $\M$ be a finite-dimensional Hadamard manifold of bounded sectional curvature satisfying \cref{assump:bounded_curvature_tensor}. Consider $f:\NN\subset\M\to\R$ be an $\L$-smooth and $\mu$-strongly g-convex differentiable function in $\ball(\xastg, 3\RR)$, where $\xastg$ is its global minimizer and where $\RR \geq \dist(\xInit, \xastg)$ for an initial point $\xInit$. For any $\epsilon > 0$, \cref{alg:instance_of_riemacon} yields an $\epsilon$-minimizer after $\bigotilde{\zetar^2\sqrt{\L/\mu}\log(\L\RR^2/\epsilon)}$ calls to the gradient oracle of $\f$. By using regularization, this algorithm $\epsilon$-minimizes the g-convex case ($\mu=0$) after $\bigotilde{\zetar^2\sqrt{\zetar+\L \RR^2/\epsilon}}$ gradient oracle calls.
\end{theorem}

In sum, the algorithm enjoys the same rates as \AGD{} in the Euclidean space up to a factor of $\zetar^2 = \RR^2\abs{\kmin}\coth^2(\RR\sqrt{\abs{\kmin}}) \leq (1+\RR\cdot\abs{\kmin})^2$ (our geometric penalty) and up to universal constants and $\log$ factors. Note that as the minimum curvature $\kmin$ approaches $0$ we have $\zetar \to 1$. 

Also, we emphasize that our \cref{alg:accelerated_gconvex} only needs to query the gradient of $\f$ at points in $\X$ and the $\L $-smoothness and $\mu$-strong g-convexity of $\f$ only need to hold in $\X$. This is relevant because in Riemannian manifolds the condition number $\L /\mu$ can have a lower bound depending on the size of the set, cf. \citep[Proposition 28]{martinez2020global}. Intuitively, although there are twice differentiable functions defined over the Euclidean space whose Hessian is constant everywhere, in other Riemannian cases the metric may preclude having such global condition and the larger the set is the larger the minimum possible condition number becomes. Compare this, for instance, with the bounds on the Hessian's eigenvalues of the squared-distance function in \cref{fact:hessian_of_riemannian_squared_distance}. We also note that we can reduce the geometric penalties with acces to a convex projection oracle, cf. \cref{sec:other_subroutine}.

\section{Related work and comparisons}\label{sec:related_work}

We compare our results with previous works. We have summarized most of the following discussion in \cref{table:comparisons:riemannian}. We include Nesterov's \AGD{} in the table for comparison purposes\footnote{Note that the original method in \citep{nesterov1983method} needed to query the gradient of the function outside of the feasible set, and this was later improved to only require queries at feasible points \citep{nesterov2005smooth} as in our work, hence our choice of citation in the table.}.   
There are some works on Riemannian acceleration that focus on empirical evaluation or that work under strong assumptions \citep{liu2017accelerated, alimisis2019continuous, huang2019extending, alimisis2020practical, lin2020accelerated}, see \citep{martinez2020global} for instance for a discussion on these works. We focus the discussion on the most related work with guarantees. \citep{zhang2018towards} obtain an algorithm that, up to constants, achieves the same rates as \AGD{} in the Euclidean space, for $\L$-smooth and $\mu$-strongly g-convex functions but only \textit{locally}, namely when the initial point starts in a small neighborhood $N$ of the minimizer $\xastg$: a ball of radius $O((\mu/\L)^{3/4})$ around it. \citep{ahn2020nesterov} generalize the previous algorithm and, by using similar ideas as in \citep{zhang2018towards} for estimating a lower bound on $\f$, they adapt the algorithm to work globally, proving that it eventually decreases the objective as fast as \AGD{}. However, as \citep{martinez2020global} noted, it takes as many iterations as the ones needed by Riemannian gradient descent (\newtarget{def:acronym_riemannian_gradient_descent}{\RGD{}}) to reach the neighborhood of the previous algorithm. The latter work also noted that in fact \RGD{} and the algorithm in \citep{zhang2018towards} can be run in parallel and combined to obtain the same convergence rates as in \citep{ahn2020nesterov}, which suggested that for this technique, full acceleration with the rates of \AGD{} only happens over the small neighborhood $N$ in \citep{zhang2018towards}. Note however that \citep{ahn2020nesterov} show that their algorithm will decrease the function value faster than \RGD{}, but this is not quantified.  \citep{jin2021riemannian} developed a different framework, arising from \citep{ahn2020nesterov} but with the same guarantees for accelerated first-order methods. We do not feature it in the table. \citep{criscitiello2022negative} showed, under mild assumptions, that in a ball of center $x\in\M$ and radius $O((\mu/\L)^{1/2})$ containing $\xastg$, the pullback function $\f\circ\expon{x}:\Tansp{x}\M\to\R$ is Euclidean, strongly convex, and smooth with condition number $O(\L/\mu)$, so \AGD{} yields local acceleration as well. In short, acceleration is possible in a small neighborhood because there the manifold is almost Euclidean and the geometric deformations are small in comparison to the curvature of the objective. These techniques fail for the g-convex case since the neighborhood becomes a point ($\mu/\L = 0$).

Finding fully accelerated algorithms that are \textit{global} presents a harder challenge. By a fully accelerated algorithm we mean one with rates with same dependence as \AGD{} on $\L $, $\epsilon$, and if it applies, on $\mu$. \citet{martinez2020global} provided such algorithms for g-convex functions, strongly or not, defined over manifolds of constant sectional curvature and constrained to a ball of radius $\RR$. The convergence rates initially had large constants with respect to $\RR$ but were later improved, cf. \cref{table:comparisons:riemannian}. \citet{kim2022accelerated} designed global algorithms with the same rates as \AGD{} up to universal constants and a factor of $\zetad$, their geometric penalty. However, they need to assume that the iterates of their algorithm remain in their feasible set $\X$ and they point out on the necessity of removing such an assumption, which they leave as an open question. Our work solves this question for a wide class of Hadamard manifolds. In their technique, they show they can use the structure of the accelerated scheme to \textit{move} lower bound estimations on $\f(\xastg)$ from one particular tangent space to another without incurring extra errors, when the right Lyapunov function is used. By \textit{moving} lower bounds here we mean finding suitable lower bounds that are simple (a quadratic in their case), when pulled-back to one tangent space, if we start with a similar bound that is simple when pulled-back to another tangent space.

\paragraph{Lower bounds.}
In this paragraph, we omit constants depending on the curvature bounds in the {$\text{big-}O$} notations for simplicity. \citep{hamilton2021no} proved an optimization lower bound showing that acceleration in Riemannian manifolds is harder than in the Euclidean space. \citep{criscitiello2022negative} largely generalized their results. They essentially show that for a large family of Hadamard manifolds, there is a function that is smooth and strongly g-convex in a ball of radius $\RR$ that contains the minimizer $\xastg$, and for which finding a point that is $\RR/5$ close to $\xastg$ requires $\bigomegatilde{\RR}$ calls to the gradient oracle. Note that these results do not preclude the existence of a fully accelerated algorithm with rates $\bigotilde{\RR} +$\AGD{} rates, for instance. A similar hardness statement is provided for smooth and only g-convex functions. Also, reductions as in \citep{martinez2020global} evince this hardness is also present in this case. 

\paragraph{Handling constraints to bound geometric penalties.}\hypertarget{sec:handling_constraints_to_bound_geometric_penalties}{} In our algorithm and in all other known fully accelerated algorithms, learning rates depend on the diameter of the feasible set. This is natural: estimation errors due to geometric deformations depend on the diameter via the constants $\zeta{\D}$, $\delta{\D}$, the cosine-law Riemannian inequalities \cref{lemma:cosine_law_riemannian}, or other analogous inequalities, and the algorithms take these errors into account. All other previous works are not able to deal with any constraints and hence they simply assume that the iterates of their algorithms stay within one such pre-specified set, except for \citep{martinez2020global} and \citep{criscitiello2022negative} that enforce a ball constraint, as we explained above. However, these two works have their applicability limited to spaces of constant curvature and to local optimization, respectively. Note that even if one could show that given a choice of learning rate, convergence implies that the iterates will remain in some compact set, then because the learning rates depend on the diameter of the set, and the diameter of the set would depend on the learning rates, \emph{one cannot conclude from this argument that the assumption these works make is going to be satisfied}. In contrast, in this work, we design a general accelerated framework and an instance of it that keep the iterates bounded in a set we \emph{pre-specify}, effectively bounding geometric penalties while we do not need to resort to any other extra assumptions, solving the open question in \citep{kim2022accelerated}.

Some other works study and use Riemannian metric projections, see \citep{walter1974metric, hosseini2013metric, barani2013metric, bacak2014convex, zhang2016first} and references therein. Among them, \citep{zhang2016first} introduced several deterministic and stochastic first-order methods that use metric-projection oracles.

\paragraph{Riemannian proximal methods.} There are some works that study proximal methods in Riemannian manifolds, but most of them focus on asymptotic results or assume the proximal operator can be computed exactly \citep{wang2015convergence, bento2017iteration, bento2016new, khammahawong2021tseng, chang2021inertial}. The rest of these works study proximal point methods under different inexact versions of the proximal operator as ours and they do not show how to implement their inexact version in applications, like in our case of smooth and g-convex optimization. In contrast, we implement the inexact proximal operator with a first-order method. \citep{ahmadi2014convergence} provide a convergence analysis of an inexact proximal point method but when applied to optimization they assume the computation of the proximal operator is exact. \citep{tang2014rate} uses a different inexact condition and proves linear convergence, under a growth condition on $\f$. \citep{wang2016proximal} obtains linear convergence of an inexact proximal point method under a different growth assumption on $\f$ and under an absolute error condition on the proximal function. Most importantly, none of these methods presented acceleration.

\section{Conclusion and future directions}\label{sec:future_work}
In this work, we pursued an approach that, by designing and making use of inexact Riemannian proximal methods, yielded accelerated optimization algorithms. Consequently we were able to work without an undesirable assumption that most previous methods required, whose potential satisfiability is not clear: that the iterates stay in certain specified geodesically-convex set without enforcing them to be in the set. A future direction of research is the study of whether there are algorithms like ours that incur even lower geometric penalties or that do not incur $\log(1/\epsilon)$ factors. Determining whether the convergence rates of fully accelerated algorithms necessarily incur a geometric factor is an interesting open problem: current lower bounds only require an additive geometric penalty and the rate of unaccelerated unconstrained \RGD{} in \citep[Thm. 15]{zhang2016first} does present an additive geometric constant only, while all known accelerated methods have a multiplicative geometric constant in their rates. Note that for the local algorithms in \cref{table:comparisons:riemannian}, this factor is a constant. Another interesting direction consists of studying generalizations of our approach to more general manifolds, namely the full Hadamard case, and manifolds of non-negative or even of bounded sectional cuvature.

\clearpage

\acks{
We thank Christophe Roux for proofreading and bringing an error to our attention in an early version of this draft. David Martínez-Rubio was partially funded by the DFG Cluster of Excellence MATH+ (EXC-2046/1, project id 390685689) funded by the Deutsche Forschungsgemeinschaft (DFG).
}

\printbibliography[heading=bibintoc] 

\clearpage

\appendix

\section[Convergence of Riemacon (Algorithm \ref{alg:accelerated_gconvex})]{Convergence of \riemacon{} (\texorpdfstring{\cref{alg:accelerated_gconvex}}{Algorithm \ref{alg:accelerated_gconvex}})}
We start the analysis by noting a property that our parameters satisfy.
\begin{lemma}\label{lemma:inequalities_on_ak}
    For the parameter choices of $\ak$ and $\Ak[k-1]$ in \cref{alg:accelerated_gconvex} we have, for all $k\geq 1$:
    \[
        \frac{8\lambda}{9} (\xi \Ak[k-1] + \ak)  \geq \ak[][2] \geq  \frac{3\lambda}{4}(\xi\Ak[k-1]+\xi\ak).
    \] 
\end{lemma}
\begin{proof}
    It is a simple computation to check that $\ak$ and $\Ak[k-1]$ satisfy such inequality. The inequalities are equivalent to the following, which trivially holds:
\begin{align*}
 \begin{aligned}
     \frac{8}{9}( (k^2-k+64k\xi &-64\xi +1000\xi^2 ) + (2k + 64\xi) ) \geq \frac{4}{5}(k^2+64k\xi + 1024\xi^2) \\
        &\geq \frac{3}{4}((k^2-k+64k\xi -64\xi +1000\xi^2 ) + (2k\xi + 64\xi^2)).
   \end{aligned}
\end{align*}
\end{proof}

We now prove \cref{thm:psi_is_lyapunov}, which will allow us to use $\psik$ as a Lyapunov function to show the final convergence rates. The proof will use \cref{lemma:the_suitable_approximate_prox_has_the_property_needed_in_alg1} and \cref{lemma:compensated_geometric_penalty}, that we state and prove afterwards. 

\begin{proof}\textbf{of \cref{thm:psi_is_lyapunov}}. \linkofproof{thm:psi_is_lyapunov}
    Inequality $(1-\Deltak)\psik \leq \psik[k-1]$ is equivalent to 
\begin{align*}
 \begin{aligned}
     (1-\Deltak)&\left(\Ak[k](\f(\yk)-\f(\xast))+ \frac{1}{2}\norm{\zy-\xast}_{\yk}^2 +\frac{\xi-1}{2}\norm{\yk-\zy}_{\yk}^2 \right)  \\
     \leq \Ak[k-1]&(\f(\yk[k-1])-\f(\xast))+ \left(\frac{1}{2}\norm{\zy[k-1][k-1]-\xast}_{\yk[k-1]}^2 +\frac{\xi-1}{2}\norm{\yk[k-1]-\zy[k-1][k-1]}_{\yk[k-1]}^2 \right)  \\
   \end{aligned}
\end{align*}
    which, by bounding $(1-\Deltak)(\f(\yk)-\f(\xast)) \leq \f(\yk)-\f(\xast)$ and reorganizing, is implied by the following:
\begin{align*}
 \begin{aligned}
     \Ak[k-1]&(\f(\yk)-\f(\yk[k-1]))+\frac{\ak}{\xi}(\f(\yk[k])-\f(\xast)) \leq \frac{1}{2} \norm{\zy[k-1][k-1]-\xast}_{\yk[k-1]}^2 - \frac{1-\Deltak}{2} \norm{\zy-\xast}_{\yk}^2 \\
     &+ \frac{\xi-1}{2}\left( \norm{\yk[k-1]-\zy[k-1][k-1]}_{\yk[k-1]}^2-(1-\Deltak)\norm{\yk-\zy}_{\yk}^2 \right). 
   \end{aligned}
\end{align*}
    We have that due to the projection in Line \ref{line:dual_projecting_step}, then $\xk$ is not very far from any $p\in\X$:
    \begin{equation}\label{eq:x_k_is_not_far_from_X}
    \dist(\xk, p) \leq \norm{\xk-\yk[k-1]}_{\yk[k-1]}+\dist(\yk[k-1], p) \circled{1}[<] \norm{\zybar[k-1][k-1]- \yk[k-1]}_{\yk[k-1]} + \D \circled{2}[\leq] 2\D,
    \end{equation}
    where $\circled{1}$ holds by the definition of $\xk$ and the fact $\yk[k-1], p \in \X$, and $\circled{2}$ is due to the projection defining $\zybar[k-1][k-1]$. Now we use the first part of \cref{lemma:the_suitable_approximate_prox_has_the_property_needed_in_alg1} with both $x\gets \yk[k-1]$ and $x\gets \xast$ and we bound the resulting errors $\errork(\cdot)$ by using the second part of \cref{lemma:the_suitable_approximate_prox_has_the_property_needed_in_alg1}. We also use \cref{lemma:compensated_geometric_penalty}, so it is enough to prove the following:
\begin{align*}
 \begin{aligned}
     \Ak[k-1]&\innp{\vkx[k], \xk-\yk[k-1]}+(\ak/\xi)\innp{\vkx[k], \xk -\zx[k-1]+\zx[k-1] -\xast} - \frac{4\lambda}{9}(\Ak[k-1]+\ak/\xi)\norm{\vkx[k]}^2 \\
     &\leq \frac{1}{2} \norm{\zx[k-1]-\xast}_{\xk}^2 - \frac{1}{2} \norm{\zx-\xast}_{\xk}^2 + \frac{\xi-1}{2}\left( \norm{\xk[k]-\zx[k-1]}_{\xk}^2-\norm{\xk-\zx}_{\xk}^2 \right) ,
   \end{aligned}
\end{align*}
    Note that thanks to \cref{lemma:compensated_geometric_penalty} now we have the potentials on the right hand side as expressions in the tangent space of $\xk$. Also, note that we have canceled some potentials proportional to $\Deltak$ coming from the bound on the error $\errork(\cdot)$. Now we use that by definition of $\xk$ we have, for all $v \in \Tansp{\xk}\M$, $\Ak[k-1]\innp{v, \xk-\yk[k-1]} = -\ak\innp{v, \xk-\zx[k-1]}$, so we use this fact for $v = \vkx[k]$ and use the following identity, that holds by the definion of $\zx\defi\zx[k-1]-\etak  \vkx[k]$: 
\[
    \frac{\ak/\xi}{\etak }\innp{\etak  \vkx[k], \zx[k-1] -\xast} = \frac{\ak/\xi}{2\etak }\left(\etak[][2]\norm{\vkx[k]}_{\xk}^2 + \norm{\zx[k-1]-\xast}_{\xk}^2 - \norm{\zx[k] -\xast}_{\xk}^2\right).
\] 
    so that, after canceling terms, it is enough to prove:
\begin{align}\label{ineq:aux_asdf}
 \begin{aligned}
     \ak(1-1/\xi)&\innp{-\vkx[k], \xk-\zx[k-1]} -\frac{\ak(1-1/\xi)}{2\etak }\etak[][2]\norm{\vkx[k]}^2  \\ 
     & + \norm{\vkx[k]}^2(-\frac{4}{9}(\Ak[k-1]\lambda+\ak\lambda/\xi) +\frac{\ak\etak }{2}) \\
     &\leq \frac{\xi-1}{2}\left( \norm{\xk[k]-\zx[k-1]}_{\xk}^2-\norm{\xk-\zx}_{\xk}^2 \right) ,
   \end{aligned}
\end{align}
    Now we show that in the previous inequality \eqref{ineq:aux_asdf}, the first line cancels with the last line. Note that $(\ak(1-1/\xi))/\etak  = (1-1/\xi)/(1/\xi) = \xi-1$. Thus, by using again the definition of $\zx[k]$, we have:
\[
    \frac{\ak(1-1/\xi)}{\etak }\innp{-\etak \vkx[k],\xk-\zx[k-1]} = \frac{\ak(1-1/\xi)}{2\etak }\left(\etak[][2]\norm{\vkx[k]}_{\xk}^2+\norm{\xk[k]-\zx[k-1]}_{\xk}^2-\norm{\xk-\zx}_{\xk}^2\right).
\] 
    Finally, it only remains to prove:
\begin{equation}\label{ineq:rate_of_as}
    \frac{\norm{\vkx[k]}^2}{2\xi}\cdot\left(-\frac{8}{9}(\xi\Ak[k-1]\lambda+\ak\lambda) + \ak[k][2]\right) \leq 0,
\end{equation}
    which holds by \cref{lemma:inequalities_on_ak}.
\end{proof}

We now show the two auxiliary lemmas that we used in the previous proof.

\begin{lemma}\label{lemma:the_suitable_approximate_prox_has_the_property_needed_in_alg1}
    Let $\hk(x) \defi \f(x) + \frac{1}{2\lambda}\dist(\xk, x)^2$ be the strongly g-convex function used at step $k$, and let $\ykast = \argmin_{y\in\X} \hk(y)$. Then, for $\yk\in\X$, if we let $\vkx[k] \defi -\exponinv{\xk}(\yk)/\lambda$, then the following holds, for all $x\in\X$:
    \[
        \f(x) \geq \f(\yk) + \innp{\vkx[k], x-\xk}_{\xk} + \frac{\lambda}{2}\norm{\vkx[k]}^2 - \errork(x) 
    \] 
     where $\newtarget{def:error_k_from_approximating_grad}{\errork}(x) \defi -\frac{1}{\lambda}\innp{\yk - \ykast, x-\xk}_{\xk} + (\hk(\yk)-\hk(\ykast))$. Moreover, if $\yk$ satisfies
     \[
         \hk(\yk) - \hk(\ykast) \leq \frac{\Deltak \dist(\xk,\ykast)^2}{78\lambda},
     \] 
     then we have 
\begin{align*}
 \begin{aligned}
     -\frac{\lambda}{2}&\norm{\vkx[k]}^2(\Ak[k-1] + \ak/\xi) + \ak\errork(\xast)/\xi + \Ak[k-1]\errork(\yk[k-1]) \\
     &\leq -\frac{4\lambda\norm{\vkx[k]}^2}{9}(\Ak[k-1] + \ak/\xi) + \frac{\Deltak}{2}\left(\norm{\xast-\zx[k-1]}_{\xk}^2 +(\xi-1)\norm{\xk-\zx[k-1]}_{\xk}^2 \right). \\
   \end{aligned}
\end{align*}

\end{lemma}

\begin{proof}
    The function $\hk$ is $\frac{1}{\lambda}$-strongly g-convex because by \cref{fact:hessian_of_riemannian_squared_distance} the function $\frac{1}{2}\dist(\xk, x)^2$ is $1$-strongly g-convex in a Hadamard manifold.  
    
    By the first-order optimality condition of $\hk$ at $\ykast$ we have that $\newtarget{def:v_tilde_k_sup_y}{\vtky} \defi \lambda^{-1}\exponinv{\ykast}(\xk)\in\partial (\f+\indicator{\X})(\ykast)$ is a subgradient of $\f+\indicator{\X}$ at $\ykast$. Thus, we have, for all $x\in\X$ and for $\newtarget{def:v_tilde_k_sup_x}{\vtkx} \defi \Gamma{\ykast}{\xk}(\vtky)$:
\begin{align*}
 \begin{aligned}
     \f(x) &\circled{1}[\geq] \f(\ykast) + \innp{\vtky, x-y_{k}^\ast}_{\ykast} \\
     & \circled{2}[\geq] \f(\ykast) + \innp{\vtkx, x-\xk}_{\xk} + \lambda\norm{\vtkx}^2 \\
     & \circled{3}[=] \f(\yk) + \innp{\vkx[k], x-\xk}_{\xk} + \frac{\lambda}{2}\norm{\vkx[k]}^2 + \frac{\lambda}{2}\norm{\vtkx}^2 \\
     & \ \  + \innp{\vtkx - \vkx[k], x-\xk}_{\xk}+ \left((\f(\ykast) + \frac{\lambda}{2}\norm{\vtkx}^2 ) - (\f(\yk) + \frac{\lambda}{2}\norm{\vkx[k]}^2) \right) \\
     & \circled{4}[\geq] \f(\yk) + \innp{\vkx[k], x-\xk}_{\xk} + \frac{\lambda}{2}\norm{\vkx[k]}^2  + \frac{1}{\lambda}\innp{\yk - \ykast, x-\xk}_{\xk} -(\hk(\yk)-\hk(\ykast)) \\
   \end{aligned}
\end{align*}
    where $\circled{1}$ holds because $\vtky \in \partial (\f+\indicator{\X})(y_{k}^\ast)$ and $x, \ykast \in\X$. In $\circled{2}$, we used the first part of \cref{lemma:moving_hyperplanes:exact_approachment_recession} along with $\deltad=1$. We just added and subtracted some terms in $\circled{3}$, and in $\circled{4}$, we dropped $\frac{\lambda}{2}\norm{\vtkx}^2$, and we used the definitions of $\hk$, $\vtkx$, and $\vkx[k] = -\exponinv{\xk}(\yk)/\lambda$.

    Now we proceed to prove the second part. The following holds:
\begin{align}\label{eq:bounding_part_of_the_not_really_subgradient_error}
 \begin{aligned}
     -\frac{\ak}{\lambda\xi}&\innp{\yk - \ykast, \xast-\xk}_{\xk} - \Ak[k-1]\frac{1}{\lambda}\innp{\yk - \ykast, \yk[k-1]-\xk}_{\xk} \\
     & \circled{1}[\leq] \frac{1}{\lambda}\norm{\yk - \ykast}_{\xk}\cdot \norm{\frac{\ak}{\xi} \xast + \Ak[k-1]\yk[k-1] - (\frac{\ak}{\xi} + \Ak[k-1])\xk}_{\xk} \\
     &\circled{2}[\leq]  \frac{1}{\lambda}\dist(\yk, \ykast)\cdot \frac{\ak}{\xi}\norm{\xast-\zx[k-1]  + (\xi-1)(\xk-\zx[k-1])}_{\xk} \\
     &\circled{3}[\leq]  \frac{1}{\lambda}\sqrt{2\lambda(\hk(\yk)-\hk(\ykast))}\cdot \frac{\ak}{\xi}\sqrt{\xi}\sqrt{\norm{\xast-\zx[k-1]}_{\xk}^2 +(\xi-1)\norm{(\xk-\zx[k-1])}_{\xk}^2} \\
     &=  \sqrt{\frac{2\ak[][2](\hk(\yk)-\hk(\ykast))}{\Deltak\lambda\xi}} \cdot \sqrt{\Deltak}\sqrt{\norm{\xast-\zx[k-1]}_{\xk}^2 +(\xi-1)\norm{(\xk-\zx[k-1])}_{\xk}^2} \\
     &\circled{4}[\leq]  \frac{\ak[][2](\hk(\yk)-\hk(\ykast))}{\Deltak\lambda\xi} + \frac{\Deltak}{2}(\norm{\xast-\zx[k-1]}_{\xk}^2 +(\xi-1)\norm{(\xk-\zx[k-1])}_{\xk}^2), \\
   \end{aligned}
\end{align}
    where $\circled{1}$ groups some terms and uses Cauchy-Schwartz. In inequality $\circled{2}$, for the first term we bounded the distance between $y_{k}^\ast$ and $\yk$ estimated from $\Tansp{\xk}\M$ by the actual distance, which is a property that holds in Hadamard manifolds and it holds by the first part of \cref{corol:moving_quadratics:exact_approachment_recession} with $\deltad=1$, $p\gets \ykast$, $y\gets \yk$, $x\gets \xk$, $z^y \gets 0$. The second term is substituted by a term of equal value by using Euclidean trigonometry in $\Tansp{\xk}\M$, as in the following. Let $w \defi \frac{1}{\ak/\xi + \Ak[k-1]} (\frac{\ak}{\xi}\exponinv{\xk}(\xast) + \Ak[k-1]\exponinv{\xk}(\yk[k-1]))$ and let $u \in \Tansp{\xk}$ be the point in the line containing $\exponinv{\xk}(\yk[k-1])$ and $0 = \exponinv{\xk}(\xk)\in \Tansp{\xk}$ such that the triangle with vertices $0$, $\exponinv{\xk}(\yk[k-1])$ and $w$ and the triangle with vertices $u$, $\exponinv{\xk}(\yk[k-1])$ and $\exponinv{\xk}(\xast)$ are similar triangles, and so 
\begin{equation}\label{eq:trigonometry_in_Tx}
    \frac{\norm{\exponinv{\xk}(\xast)-u}}{\norm{w-\exponinv{\xk}(\xk)}} \circled{5}[=] \frac{\norm{\exponinv{\xk}(\xast)-\exponinv{\xk}(\yk[k-1])}}{\norm{w-\exponinv{\xk}(\yk[k-1])}} \circled{6}[=] \frac{\Ak[k-1]+\ak/\xi}{\ak/\xi}.
\end{equation}
    We used the triangle similarity in $\circled{5}$ and in $\circled{6}$ we used the definition of $w$ as a convex combination of $\exponinv{\xk}(\xast)$ and $\exponinv{\xk}(\yk[k-1])$. It is enough to show $u = \xi \zx[k-1]$ as in such a case what we applied in $\circled{2}$ is equivalent to the equality \eqref{eq:trigonometry_in_Tx} above. By the definition of $\xk$, we have $\circled{7}$ below and by triangle similarity we have $\circled{8}$ below: 
\[
    \zx[k-1] \circled{7}[=] -\frac{\Ak[k-1]}{\ak}\exponinv{\xk}(\yk[k-1]) \circled{8}[=]  \frac{\Ak[k-1]}{\ak}\cdot \frac{\ak/\xi}{\Ak[k-1]}u = \frac{1}{\xi} u,
\] 
as desired. In the next inequality $\circled{3}$, we used that by $(1/\lambda)$-strong g-convexity of $\hk$ and by optimality of $\ykast$, we have $\frac{1}{2\lambda}\dist(\cdot, \ykast)^2 \leq \hk(\cdot) - \hk(\ykast)$. For the second term, we used that for vectors $b, c \in \R^n$ and $\omega \in \R_{\geq 0}$, we have, by Young's inequality, $\norm{b+wc} = \sqrt{\norm{b}^2 + \omega^2 \norm{c}^2 + 2\innp{\sqrt{\omega}b, \sqrt{\omega}c}} \leq \sqrt{(1+\omega)(\norm{b}^2 + \omega \norm{c}^2)}$. In $\circled{4}$ we used Young's inequality.

    Before we conclude, we note that 
    \begin{equation}\label{eq:aux_bounding_distances_from_x_to_both_prox_and_aprox_prox}
        \dist(\xk, \ykast) \leq \sqrt{2} \dist(\xk, \yk),
    \end{equation}
    which is implied by the following, where we use the same as in $\circled{3}$ above, the assumption on $\yk$ and $\Deltak \leq 1$:
\begin{align*}
 \begin{aligned}
     \dist(\xk, \ykast) &\leq \dist(\xk, \yk) + \dist(\yk, \ykast) \leq \dist(\xk, \yk) + \sqrt{2\lambda(\hk(\yk)-\hk(\ykast))} \\ 
     &\leq \dist(\xk, \yk) + \sqrt{\Deltak/34}\cdot \dist(\xk, \ykast) \leq  \dist(\xk, \yk) + \dist(\xk, \ykast)/4.
   \end{aligned}
\end{align*}

    Finally, we can make use of \eqref{eq:bounding_part_of_the_not_really_subgradient_error} and \eqref{eq:aux_bounding_distances_from_x_to_both_prox_and_aprox_prox} to obtain the claim in the second part of the lemma:
\begin{align*}
 \begin{aligned}
     -\frac{\lambda}{2}&\norm{\vkx[k]}^2(\Ak[k-1] + \ak/\xi) + \ak\errork(\xast)/\xi + \Ak[k-1]\errork(\yk[k-1]) - \frac{\Deltak}{2}\norm{\xast-\zx[k-1]}_{\xk}^2 \\
     & \quad  - \Deltak\frac{\xi-1}{2}\norm{(\xk-\zx[k-1])}_{\xk}^2  \\
     &\leq -\frac{\lambda}{2}\norm{\vkx[k]}^2(\Ak[k-1] + \ak/\xi) + \left(\Ak[k-1] + \ak/\xi + \frac{\ak[][2]}{\Deltak\lambda\xi} \right) (\hk(\yk)-\hk(\ykast)) \\
     &\circled{1}[\leq] -\frac{\lambda}{2}\norm{\vkx[k]}^2(\Ak[k-1] + \ak/\xi) + (\Ak[k-1] + \ak/\xi)\left( 1 + \frac{\ak[][2]}{(\xi\Ak[k-1] + \ak)\lambda} \right) \frac{\dist(\xk,\yk)^2}{34\lambda}\\ 
     &\circled{2}[\leq] -\frac{\lambda}{2}\norm{\vkx[k]}^2(\Ak[k-1] + \ak/\xi) + \frac{\dist(\xk,\yk)^2}{18\lambda}(\Ak[k-1] + \ak/\xi)\\ 
     & \circled{3}[=] -\frac{4\lambda\norm{\vkx[k]}^2}{9}(\Ak[k-1] + \ak/\xi),
   \end{aligned}
\end{align*}
    where $\circled{1}$ holds by the assumption on $\yk$, $\Deltak \leq 1$, and \eqref{eq:aux_bounding_distances_from_x_to_both_prox_and_aprox_prox}. Inequality $\circled{2}$ uses the upper bound on $\ak[][2]$ in \cref{lemma:inequalities_on_ak}, and $\circled{3}$ uses the definition $\vkx[k] \defi -\exponinv{\xk}(\yk)/\lambda$.

\end{proof}

The following lemma allows to \textit{move} the regularized lower bounds on the objective without incurring extra geometric penalties.

\begin{lemma}[Translating Potentials with no Geometric Penalty]\label{lemma:compensated_geometric_penalty}
    Using the variables in \cref{alg:accelerated_gconvex}, for any $\Deltak \in [0, 1)$, we have  
\begin{align*}
 \begin{aligned}
     \norm{\zx[k-1]&-\xast}_{\xk}^2 - (1-\Deltak)\norm{\zx-\xast}_{\xk}^2 + (\xi-1)\left( \norm{\xk[k]-\zx[k-1]}_{\xk}^2-(1-\Deltak)\norm{\xk-\zx}_{\xk}^2 \right) \\
     &\leq \norm{\zy[k-1][k-1]-\xast}_{\yk[k-1]}^2 - (1-\Deltak)\norm{\zy-\xast}_{\yk}^2 \\
     & \ \ + (\xi-1)\left( \norm{\yk[k-1]-\zy[k-1][k-1]}_{\yk[k-1]}^2-(1-\Deltak)\norm{\yk-\zy}_{\yk}^2 \right).
   \end{aligned}
\end{align*}
    
\end{lemma}

\begin{proof}
    Firstly, by the projection step in Line \ref{line:dual_projecting_step}, we have 
    \begin{equation}\label{eq:aux_translating_w_o_penalty_1}
        \norm{\zy[k-1][k-1] - \xast}_{\yk}^2 \geq \norm{\zybar[k-1][k-1] - \xast}_{\yk}^2 \quad \quad  \text{ and } \quad \quad  (\xi-1)\norm{\zy[k-1][k-1]}_{\yk}^2 \geq (\xi-1)\norm{\zybar[k-1][k-1]}_{\yk}^2 
    \end{equation}
since the operation is a simple Euclidean projection onto the closed ball $\ball(0,\D)$ in $\Tansp{\yk}\M$ .
    By the second part of \cref{corol:moving_quadratics:exact_approachment_recession}, $y=\xk$ and $x=\yk[k-1]$ and by \eqref{eq:x_k_is_not_far_from_X}, we have $\circled{1}$ below
\begin{align}\label{eq:aux_translating_w_o_penalty_2}
 \begin{aligned}
     \norm{\zybar[k-1][k-1]&-\xast}_{\yk[k-1]}^2 + (\xi-1)\norm{\zybar[k-1][k-1]}_{\yk[k-1]}^2 \circled{1}[\geq] 
     \norm{\zx[k-1]-\xast}_{\xk}^2  + (\zeta{2\D}-1)\norm{\zx[k-1]}_{\xk}^2 + (\xi-\zeta{2\D})\norm{\zybar[k-1][k-1]}_{\yk[k-1]}^2 \\
     &\circled{2}[\geq] \norm{\zx[k-1]-\xast}_{\xk}^2  + (\xi-1)\norm{\zx[k-1]}_{\xk}^2 + (\xi-\zeta{2\D})\left(\left( \frac{\Ak[k-1]+\ak}{\Ak[k-1]}\right)^2-1\right)\norm{\zx[k-1]}_{\xk}^2 \\
     &\circled{3}[\geq] \norm{\zx[k-1]-\xast}_{\xk}^2  + (\xi-1)\norm{\zx[k-1]}_{\xk}^2 + \frac{3(\xi-1)}{2}\left(\frac{1}{1-\tauk}-1\right)\norm{\zx[k-1]}_{\xk}^2,
 \end{aligned}
\end{align}
    and $\circled{2}$ uses the definition of $\xk$. In $\circled{3}$, we used the definition of $\xi = 4\zeta{2\D} - 3$ that implies $\xi-\zeta{2\D} \geq \frac{3}{4}(\xi-1)$ and for $\newtarget{def:aux_tau_k}{\tauk}\defi\ak/(\ak+\Ak[k-1])$ we have  that $(1+\frac{\ak}{\Ak[k-1]})^2-1\geq \frac{2\ak}{\Ak[k-1]} = 2(\frac{1}{1-\tauk}-1)$. Now, using the second part of \cref{lemma:moving_quadratics:inexact_approachment_recession} with $y=\yk$, $x=\xk$ $z^x=-\etak \vkx$, $a^x = \zx[k-1]$, so that $z^x+a^x = \zx$ and $z^y + a^y = \zy$ and 
    \begin{equation}\label{ineq:r_is_less_than_1}
        r=\frac{\norm{\exponinv{\xk}(\yk)}}{\norm{z^x}} = \frac{\lambda\norm{\vkx[k]}}{\etak \norm{\vkx[k]}}  = \frac{\xi\lambda}{\ak} = \frac{5\xi}{2k+64\xi} < 5/6 < 1. 
    \end{equation}
    Note that by the choice of parameters and the fact that $r<1$, the assumptions in \cref{lemma:moving_quadratics:inexact_approachment_recession} are satisfied. Thus, the following holds
\begin{align}\label{eq:aux_translating_w_o_penalty_3}
 \begin{aligned}
     \norm{\zx-\xast}_{\xk}^2 + (\xi-1)\norm{\zx}_{\xk}^2 + \frac{\xi-1}{2}\left(\frac{r}{1-r}\right)\norm{\zx[k-1]}^2 \geq \norm{\zy-\xast}_{\yk}^2 + (\xi-1)\norm{\zy}_{\yk}^2.
 \end{aligned}
\end{align}
    Hence, combining \eqref{eq:aux_translating_w_o_penalty_1}, \eqref{eq:aux_translating_w_o_penalty_2} and \eqref{eq:aux_translating_w_o_penalty_3} we obtain that it is enough to prove
\[
    -(1-\Deltak)\left(\frac{r}{1-r}\right) + 3\left(\frac{1}{1-\tauk}-1\right) \geq 0,
\] 
    The proof will be finished if we prove the result for $\Deltak=0$. If we use this last inequality, and the fact that for $r\leq 5/6$, we have $\frac{r}{1-r} \leq 3\left(\frac{1}{1-3r/4}-1\right)$, we deduce that it suffices to show $\tauk \geq \frac{3}{4}r$ to conclude
    \[
 \frac{r}{1-r}  \leq  3\left(\frac{1}{1-3r/4}-1\right) \leq 3\left(\frac{1}{1-\tauk}-1\right).
    \] 
    Such inequality, namely $\tauk \geq \frac{3}{4}r$, is equivalent to $\frac{\ak[k][2]}{\lambda} \geq \frac{3\xi}{4}(\ak+ \Ak[k-1])$ and it holds by \cref{lemma:inequalities_on_ak}.
    
\end{proof}

Finally, we use \cref{thm:psi_is_lyapunov} to show the final convergence rates. The proof will also use \cref{lemma:where_prox_goes} that is stated and proved after the proof of the theorem.

\begin{proof}\textbf{of \cref{thm:g_convex_acceleration}}. \linkofproof{thm:g_convex_acceleration}
        Given the inequality $(1-\Deltak)\psik \leq \psik[k-1]$, proven in \cref{thm:psi_is_lyapunov}, we can use $\psik$ as a Lyapunov function in order to prove convergence rates of \cref{alg:accelerated_gconvex}. It follows straightforwardly by definition of $\psik$, in the following way
    \begin{align*}
     \begin{aligned}
         \f(\yk)-\f(\xast) &\leq \frac{\psik}{\Ak} \leq \prod_{i=1}^{k}(1-\Deltak[i])^{-1}\frac{\psik[0]}{\Ak} \circled{1}[\leq] \frac{2\psik[0]}{\Ak} \circled{2}[=] \bigol{\frac{\RInit[2]}{\lambda}\left(\frac{\Ak[0]}{\Ak} + \frac{\lambda}{\Ak}\right)} \\
         &= \bigol{\frac{\RInit[2]}{\lambda}\left(\frac{\xi}{\frac{k^2 + \xi k}{\xi} + \xi} + \frac{1}{\frac{k^2 + \xi k}{\xi} + \xi}\right)} \\
         &= \bigol{\frac{\RInit[2]}{\lambda}\left(\frac{\xi^2}{k^2 + \xi k + \xi^2} \right)} \circled{3}[=] \bigol{\frac{\RInit[2]}{\lambda k^2}\cdot\zetad^2}.
   \end{aligned}
\end{align*}
        In $\circled{1}$, we used $\prod_{i=1}^{k} (1-\Deltak) = \prod_{i=1}^{k} \frac{i(i+2)}{(i+1)^2} = \frac{k+2}{2(k+1)} \geq \frac{1}{2}$. Now for $\circled{2}$, we note the following, which is analogous to applying smoothness of the Riemannian Moreau envelope to show that an inexact prox step serves as a warm start:
\begin{align*}
     \begin{aligned}
         \f(\yk[0]) - \f(\xast) &\leq \f(\yk[0]) + \frac{1}{2\lambda} \dist(\xInit, \yk[0])^2 - \f(\xast)   \circled{4}[\leq]  \min_{y\in\X}\{\f(y) + \frac{1}{2\lambda}\dist(\xInit, y)^2\}  + \sigmak[0]  - \f(\xast) \\
         &\circled{5}[\leq] \frac{1}{2\lambda}\dist(\xInit, \xast)^2 + \frac{1}{78\lambda} \dist(x_0, \ykast[0])^2  \circled{6}[=] \bigo{\RInit^2/\lambda}.
   \end{aligned}
\end{align*}
    We used $\sigmak[0]$-optimality of $\yk[0]$ in $\circled{4}$. In $\circled{5}$ we substituted the value of $\sigmak[0]$ and we set $y\gets \xast$ in the minimum. 
    We used \cref{lemma:where_prox_goes} for $\circled{6}$. In $\circled{2}$, we also used $\frac{\xi-1}{2}\norm{\zy[0][0]}_{\yk[0]}^2=0$ and 
\begin{align*}
     \begin{aligned}
         \norm{\zy[0][0]-\xast}_{\yk[0]} = \dist(y_0, x^\ast) \leq\dist(\yk[0], \ykast[0]) +  \dist(\ykast[0], \xast) \circled{7}[\leq] \sqrt{2\lambda(\hk[0](\yk[0])-\hk[0](\ykast[0]))} + \RInit \circled{8}[=] \bigo{\RInit},
   \end{aligned}
\end{align*}
 where $\circled{7}$ holds by   $1$-strong convexity of $\lambda\hk[0]$ and \cref{lemma:where_prox_goes}, while in $\circled{8}$ we used the $\sigmak[0]$-optimality of $\yk[0]$ and \cref{lemma:where_prox_goes}.

     In $\circled{3}$, we used $\xi = \bigo{\zetad}$ and we dropped some terms in the denominator. 
        This means that the number of iterations is $\bigo{\zetad \sqrt{\frac{\RInit[2]}{\lambda\epsilon}}}$ if we want the right hand side to be bounded by $\epsilon$.

    The algorithm and analysis for strongly g-convex and smooth functions follows directly by applying the reduction in \citep[Theorem 7]{martinez2020global} to \cref{alg:accelerated_gconvex}. We denote this algorithm by \newtarget{def:Riemacon_SC}{\riemaconsc{}}$(\X, \xInit, \f, \lambda, \epsilon)$, where $\X$ is the feasible set, $\xInit$ is the initial point, $\f$ is the function to optimize, $\lambda$ is the implicit gradient descent learning rate, and $\epsilon$ is an optional parameter specifying the desired accuracy. Although the statement of the reduction in this paper assumes an $\L$-smooth function $f:\M\to \R$ to be optimized has a global minimizer in an unconstrained problem, the same proof of this theorem works if we have a $\mu$-strongly g-convex function $\f$ defined over an open set containing a closed g-convex set $\X$ and a minimizer $\xast$ of this function restricted to $\X$. The algorithm runs the algorithm for g-convex smooth minimization for $\newtarget{def:time_alg_of_just_convex_for_reduction}{\timens}(\mu, R)$, where this is defined as the number of iterations needed by the non-strongly g-convex algorithm to reach accuracy $\mu R^2/8$ if the initial distance is upper bounded by $R$. In such a case it guaranteed that the distance to the minimizer is reduced by half, and we restart the algorithm and run it again with the initial distance parameter equal to $R/2$, and so on. This happens $\bigo{\log(\mu \RInit[2]\epsilon)}$ times if we want to achieve accuracy $\epsilon$ from an initial distance $\RInit$. Thus, the total complexity in number of iterations can be bounded by $O(\timens(\mu, \RInit)\log(\mu \RInit[2]/\epsilon))$, since all initial distances are $\leq \RInit$. In our case, since we optimize over the set $\X$ with diameter $\D$, so it is $\timens(\mu,R) = \bigo{\zetad(\lambda\mu)^{-1/2} + 1}$, and the total number of iterations is $\bigo{(\zetad (\lambda\mu)^{-1/2} + 1)\log(\mu \RInit[2]/\epsilon)}$. We note that the reverse reduction in \citep{martinez2020global} yields extra geometric penalties but this one does not. 
\end{proof}

We now state and prove the lemma we used for \cref{thm:g_convex_acceleration}.

\begin{lemma}\label{lemma:where_prox_goes}
    For all $k \geq 0$, we have $\dist(\xk,\ykast) \leq \dist(\xk, \xast)$ and $\dist(\ykast,\xast) \leq \dist(\xk, \xast)$.
\end{lemma}
\begin{proof}
    Using $\xast\in\argmin_{x\in\X}\f(x)$ and $\ykast = \argmin_{y\in\X} \hk(y)$, we have: 
    \[
        \dist(\xk, \ykast)^2 -  \dist(\xk, \xast)^2 \leq 2\lambda \f(\ykast) + \dist(\xk, \ykast)^2 - 2\lambda\f(\xast) -  \dist(\xk, \xast)^2= 2\lambda (\hk(\ykast)- \hk(\xast)) \leq 0.
    \] 
    The second inequality is deduced from $1$-strong convexity of $\lambda \hk$, which holds by \cref{fact:hessian_of_riemannian_squared_distance} since we are in a Hadamard manifold, as well as the definition of $\hk$, and the fact $\xast \in \argmin_{x\in\X}\f(x)$:
    \[
    \dist(\ykast, \xast)^2 \leq 2\lambda(\hk(\xast) - \hk(\ykast)) \leq 2\lambda\f(\xast) + \dist(\xk, \xast)^2 - 2\lambda \f(\ykast) \leq  \dist(\xk, \xast)^2.
    \] 
\end{proof}

\section[Convergence of boosted Riemacon (Algorithm \ref{alg:instance_of_riemacon})]{Convergence of boosted Riemacon (\texorpdfstring{\cref{alg:instance_of_riemacon}}{Algorithm \ref{alg:instance_of_riemacon}})} \label{sec:proofs_ball_oracle}

We start by showing that the iterates of \cref{alg:instance_of_riemacon} stay reasonably bounded, which is crucial in order to bound geometric penalties.

\begin{proposition}\label{prop:we_go_no_farther_than_2R}
    The iterates $x_k$ of \cref{alg:instance_of_riemacon} satisfy $\dist(x_k, \xastg) \leq 2\RR$.
\end{proposition}
\begin{proof}
    We first show that the optimizer $x^\ast_k$ in the ball $\X_k$ is no farther than the center of $\X_k$ to $\xastg$, that is, $\dist(x^\ast_k, \xastg) \leq \dist(x_{k-1}, \xastg)$. We assume $\xastg$ is not in the ball because otherwise the property holds trivially. The geodesic segment joining $x^\ast_k$ and $\xastg$ does not contain any other point of the ball, since otherwise by strong convexity we would have that the function value of one such point would be lower than $\f(x^\ast_k)$. This fact implies that the angle between $\exponinv{x^\ast_k}(\xastg)$ and $\exponinv{x^\ast_k}(x_{k-1})$ is obtuse, and so $\circled{1}$ holds below and by using \cref{lemma:cosine_law_riemannian} we conclude $\dist(x^\ast_k, \xastg) \leq \dist(x_{k-1}, \xastg)$:
\begin{align*}
 \begin{aligned}
     0 &\circled{1}[\geq] 2\innp{\exponinv{x_{k}^\ast}(\xastg), \exponinv{x_{k}^\ast}(x_{k-1})} \geq \dist(x_{k}^\ast, \xastg)^2 + \deltar \cdot \dist(x_k^\ast, x_{k-1})^2 - \dist(x_{k-1}, \xastg)^2 \\
     &\geq \dist(x_{k}^\ast, \xastg)^2  - \dist(x_{k-1}, \xastg)^2.
   \end{aligned}
\end{align*}

    If instead of optimizing exactly in the ball we obtain a close approximation, the iterates will not get very far from $\xastg$. Indeed, by $\mu$-strong convexity, if $x_k$ is an $\epsilonp$-minimizer of $\f$ in $\X_k$, we have that $\dist(x^\ast_k, x_k) \leq \sqrt{\frac{2\epsilonp}{\mu}} \leq \frac{\RR}{\TT}$, where we used the definition of $\epsilonp = \min\{\frac{\D\epsilon}{8\RR},\frac{\mu\RR^2}{2\TT^2}\}$ in the last inequality. Consequently, applying the non-expansiveness and this last inequality recursively, we obtain
    \[
        \dist(x_{\TT}, \xastg) \leq \dist(x_{\TT}^\ast, \xastg) + \dist(x_{\TT}^\ast, x_{\TT}) \leq \dist(x_{\TT-1}, \xastg) + \frac{\RR}{\TT} \leq \dots \leq \dist(x_{0}, \xastg) + \RR \leq 2\RR.
    \] 
\end{proof}

Before we prove \cref{thm:instantiation}, let's discuss about the initialization of $\D$. As we explain in \cref{sec:subroutine}, we can apply the subroutine in \citep[Section 6]{criscitiello2022negative} for any value of $\D$ that satisfies (notice $\D$ is twice the radius of the ball): 
\begin{equation}\label{eq:aux_D_ineq}
    \D \leq (46\RR\abs{\kmin}\zeta{D})^{-1},
\end{equation}
If $\D=2\RR$ satisfies the inequality, then the algorithm uses this value. If it is not satisfied, then for any value $\D\geq 0$ that satisfies the inequality it must be $\D < 2\RR$, so we assume that this inequality holds for the rest of the argument. Indeed, it is a consequence of the function $x^2\coth(x)$ being monotonously increasing for $x\geq 0$ and that given the definition of $\zeta{\D}=\D\sqrt{\abs{\kmin}}\coth(\D\sqrt{\abs{\kmin}})$, we have that inequality \eqref{eq:aux_D_ineq} is equivalent to $\D^2\abs{\kmin}\coth(\D\sqrt{\abs{\kmin}}) \leq (46\RR\sqrt{\abs{\kmin}})^{-1}$.  In this case, the larger $\D$ is, the faster the algorithm runs. So one could solve the $1$-dimensional problem $\D = (46\RR\abs{\kmin}\zeta{\D})^{-1}$ on $\D$ in order to obtain the best guarantee. On the other hand, we can provide the simple bound on this $1$-dimensional problem $\D = 1/(70\RR\abs{\kmin})$ which would only lose a constant in the final convergence rates. We show now how this is indeed a bound. Let $x$ be $\D \sqrt{\abs{\kmin}}$, for some $\D$ satisfying inequality \eqref{eq:aux_D_ineq} and let $S$ be the set of all such $x \geq 0$. Because we want $x^2\coth(x) \leq (46\RR\sqrt{\abs{\kmin}})^{-1}$ and the right hand side is upper bounded by $\leq 1/(23x)$, then by monotonicity it must be $S \subset [0, 1/4]$. It holds that for this interval the fourth derivative of $x^2 \coth(x) \leq 0$, which along with its third order Taylor expansion yields $\circled{1}$ below, so the points satisfying $\circled{3}$ below are in $S$ and we can use $\D = \frac{x}{\sqrt{\abs{\kmin}}} = \frac{1}{70\RR\abs{\kmin}} \leq \frac{3}{4\cdot 46\RR\abs{\kmin}}$ as our simple-to-compute bound:
\[
    x^2\coth(x) \circled{1}[\leq] x+\frac{x^3}{3} \circled{2}[\leq] \frac{4}{3}x \circled{3}[\leq] \frac{1}{46\RR\sqrt{\abs{\kmin}}}.
\] 
where in $\circled{2}$ we used $x < 1$ for all $x\in S$. Now, we can proceed to prove the theorem.

\begin{proof}\textbf{of \cref{thm:instantiation}}. \linkofproof{thm:instantiation}
    If $\D = 2\RR$, which is the case in which the condition in Line \ref{line:if_D_equal_R_one_single_Riemacon} of \cref{alg:instance_of_riemacon} is satisfied, then we just need to call \cref{alg:accelerated_gconvex} once in the corresponding ball $\ball(\xInit, \RR)$ and we obtain rates $\bigotilde{\zetar^2 \sqrt{\frac{\L}{\mu}}}$. So from now on we assume $\D < 2\RR$. Let $\TT = \lceil \frac{4\RR}{\D} \ln(\frac{\L\RR^2}{\epsilon}) \rceil$ and let $\epsilonp = \min\{\frac{\D\epsilon}{8\RR},\frac{\mu\RR^2}{2\TT^2}\}$. Since every time we call \cref{alg:accelerated_gconvex} we do it over a ball of diameter $\D$, we still use the notation $\zetad \defi \zeta{\D}$ to refer to the geometric constant associated to the sets $\X_k$, for every $k\geq 1$. Recall that we use $\zetar \defi \zeta{\RR} = \RR\sqrt{\abs{\kmin}}\coth(\RR\sqrt{\abs{\kmin}}) \in [\RR\sqrt{\abs{\kmin}}, \RR\sqrt{\abs{\kmin}}+1]$.

    By definition, it is $\D \leq (46\RR \abs{\kmin} \zetad)^{-1}$. Using $2\RR > \D$ and $\zetad \in [ \D\sqrt{\abs{\kmin}}, \D\sqrt{\abs{\kmin}}+1]$, we conclude $\D\leq 1/\sqrt[3]{46}\sqrt{\abs{\kmin}}\leq  1/\sqrt{\abs{\kmin}}$ and $\zetad \leq \D\sqrt{\abs{\kmin}} +1 \leq 2$.  Since $\zetad=\bigo{1}$, the subroutine in Line \ref{line:subroutine} takes $\bigotilde{1}$ gradient oracle calls by the analysis in \cref{sec:subroutine} and thus, Line \ref{line:riemacon_w_subroutine} of \cref{alg:instance_of_riemacon} takes $\bigotilde{\sqrt{\frac{\L}{\mu}}\log(\frac{1}{\epsilonp})}$ gradient oracle calls to optimize in the ball $\X_k$ of diameter $\D$, for any $k$.
Recall that we denote the global optimizer of $\f$ by $\xastg$. Define the g-convex combination 
\[
    \tilde{x}_{k} = \expon{x_{k-1}}\left(\frac{\D}{4\RR}\exponinv{x_{k-1}}(\xastg)\right) = \expon{x_{k-1}}\left((1-\frac{\D}{4\RR})x_{k-1} + \frac{\D}{4\RR}\xastg\right).
\] 
Since $\X_k$ is a ball of radius $\D/2$ and by \cref{prop:we_go_no_farther_than_2R}, it is $\dist(x_k, x^\ast) \leq 2\RR$, we have $\tilde{x}_{k} \in \X_k$. Consequently, we have
\[
    \f(x_{k}) \circled{1}[\leq] \f(\tilde{x}_{k}) + \epsilonp \circled{2}[\leq] (1-\frac{\D}{4\RR})\f(x_{k-1}) + \frac{\D}{4\RR} \f(\xastg) + \epsilonp,
\] 
where $\circled{1}$ is due to the guarantees of the optimization in the ball and the fact that $\tilde{x}_k \in \X_k$, $\circled{2}$ holds due to g-convexity. Subtracting $\f(\xastg)$ in both sides and rearranging, we obtain 
\[
\f(x_{k})-\f(\xastg) \leq (1-\frac{\D}{4\RR})(\f(x_{k-1}) - \f(\xastg)) + \epsilonp.
\] 
Applying this inequality recursively, we obtain
\begin{align*}
 \begin{aligned}
     \f(x_{\TT})-\f(\xastg) &\leq (1-\frac{\D}{4\RR})^{\TT}(\f(\xInit) - \f(\xastg)) + \epsilonp\sum_{i=0}^{\TT-1} (1-\frac{\D}{4\RR})^i \\
     &\circled{1}[\leq] \exp(-\frac{\D \TT}{4\RR})\frac{\L\RR^2}{2} + \frac{4\RR}{\D}\epsilonp\\
     &\leq \frac{\epsilon}{2} + \frac{\epsilon}{2} = \epsilon.
   \end{aligned}
\end{align*}
    Above, we used $1-x \leq \exp(-x)$, we used smoothness to bound $\f(\xInit) - \f(\xastg) \leq \frac{\L\dist(\xInit, \xastg)^2}{2}$, we bounded $\sum_{i=0}^{\TT-1} (1-\frac{\D}{4\RR})^i \leq \sum_{i=0}^{\infty} (1-\frac{\D}{4\RR})^i = \frac{4\RR}{\D}$ and we used the values of $\epsilonp$ and $\TT$.  Finally, we compute the complexity of this algorithm. We have $\TT$ iterations taking $\bigotilde{\sqrt{\frac{\L}{\mu}}}$ gradient oracle queries each. Using the value of $\TT$ and $\D$, we obtain that in total, we call the gradient oracle $\bigotilde{\frac{\RR}{\D} \sqrt{\frac{\L}{\mu}}} = \bigotilde{\RR^2\abs{\kmin}\sqrt{\frac{\L}{\mu}}} = \bigotilde{\zetar^2\sqrt{\frac{\L}{\mu}}}$ times, in both of the suggested initializations for $\D$, cf. \cref{alg:instance_of_riemacon}.

    We conclude by studying the case in which $\f$ is not strongly convex. Assume there is a global optimizer $\xastg$ and as before let $\RR \geq \dist(\xInit, \xastg)$. Given $\epsilon > 0$, we use the regularizer $r(x) = \frac{\epsilon}{2\RR^2}\Phi_{\xInit}(x) = \frac{\epsilon}{2\RR^2}\dist(\xInit, x)^2$. Let $x^\ast_{\epsilon}$ be the minimizer of $\f+r$. By \citep[Lemma 21]{martinez2020global}, we have $\dist(\xInit, x^\ast_{\epsilon}) \leq \dist(\xInit, \xastg) \leq \RR$.  We run \cref{alg:instance_of_riemacon} on $\f+r$, which satisfies that the iterates of the algorithm and the subroutine go no farther than $2\RR+\D/2 < 3\RR$ from $x^\ast_{\epsilon}$. Indeed, the centers of the balls $\X_k$ are at a distance at most $2\RR$ from $x^\ast_{\epsilon}$ by \cref{prop:we_go_no_farther_than_2R} and each ball has radius $\D/2$. Recall that we are still optimizing over a Hadamard manifold. So in $\ball(x^\ast_{\epsilon}, 3\RR)$, we have that $\f+r$ is strongly g-convex with constant $\frac{\epsilon}{\RR^2}$. Moreover, its smoothness constant is $\zeta{3\RR}\cdot\frac{\epsilon}{\RR^2} + \L = \bigo{\zetar\cdot\frac{\epsilon}{\RR^2} + \L} $, by \cref{fact:hessian_of_riemannian_squared_distance}. Hence, the algorithm finds an $\epsilon/2$ minimizer $x_{T'}$ of $\f+r$ after $T'=\bigotilde{\zetar^2\sqrt{\zetar + \frac{\L\RR^2}{\epsilon}}}$ queries to the gradient oracle. By definition, it is $\dist(\xInit, \xastg) \leq \RR$ so $r(\xastg) \leq \frac{\epsilon}{2\RR^2} \cdot \RR^2 = \frac{\epsilon}{2}$ and thus $x_{T'}$ is an $\epsilon$-minimizer of $\f$:
    \[
        \f(x_{T'}) \leq \f(x_{T'})+r(x_{T'}) \leq \f(\xastg)+r(\xastg) + \frac{\epsilon}{2} \leq \f(\xastg) + \epsilon.
    \] 

Finally, we note that the condition of the theorem requiring optimizing in $\ball(\xastg, 3\RR)$ can be relaxed to work in $\ball(\xastg, c_1\RR +c_2)$ for $c_1>1$ and $c_2 \in (0, 1/(70\RR\abs{\kmin}) )$, by solving the ball subproblems more accurately and in smaller balls.
\end{proof}

\subsection{Details of the subroutine chosen by \texorpdfstring{\cref{alg:instance_of_riemacon}}{Algorithm \ref{alg:instance_of_riemacon}}  for Line \ref{line:subroutine} of \texorpdfstring{\cref{alg:accelerated_gconvex}}{Algorithm \ref{alg:accelerated_gconvex}}}\label{sec:subroutine}

Given a constant $F$ such that $\norm{\nabla \curvtensor} \leq F$, in their Proposition 6.1, \citet{criscitiello2022negative} argue that given an $L'$-smooth and $\oldmu'$-strongly g-convex function in a ball of radius $r\leq \min\{\frac{\sqrt{\oldmu'}}{4\sqrt{L'}\abs{\kmin}}, \frac{\abs{\kmin}}{4F} \}$\footnote{This bound corresponds to the case of Hadamard manifolds. Their statement applies more generally to manifolds of bounded sectional curvature, in which case $\abs{\kmin}$ would be substituted by $\max\{\abs{\kmin}, \kmax\}.$},  the retraction of the function in the ball to the Euclidean space $\hat{h}(\cdot) \defi h\circ\expon{x_k}(\cdot)$ is strongly convex and smooth with condition number $\frac{3L'}{\oldmu'}$. Here, $\frac{1}{F}$ is interpreted as $+\infty$. They assume that the global minimizer is in this ball, but this fact is only used in order to use $L'$-smoothness to bound the Lipschitz constant of the function by $2rL'$. In our case, the global optimizer is at a distance at most $3\RR$ from any point in any of our balls $\X_k$, as argued in the previous section. However, we can bound the Lipschitz constant by other means. The functions we will apply this subroutine to have the form  $h(y) \defi \f(y) + \frac{1}{\lambda}\dist(x, y)^2$, where $x$ is a point such that $\dist(x, y) \leq 2D$ for all $y\in\X_k$, cf. Line \ref{line:subroutine} in \cref{alg:accelerated_gconvex} and \eqref{eq:x_k_is_not_far_from_X}. Here $\D=2r$ is the diameter of $\X_k$. Using the value of $\lambda$, we have that the smoothness of $g:y\mapsto\frac{1}{\lambda}\dist(x, y)^2$ in the ball $\X_k$ is $\L$ and the global minimizer of this function is at most a distance $2D=4r$. So we can estimate the Lipschitz constant of such an $h$ as 
\[
\max_{y\in \X_k} \norm{\nabla h(y)} \leq \max_{y\in \X_k} \norm{\nabla \f(y)} + \max_{y\in \X_k} \norm{\nabla g(y)} \leq 6\RR \L + 8r\L \leq 14\RR \L,
\] 
where the last inequality uses $r\leq \RR$ which holds by construction of \cref{alg:instance_of_riemacon}. Now, it is enough to satisfy the following inequality in Proposition 6.1 in \citep{criscitiello2022negative} in order to have that the Euclidean pulled-back function has condition number of the same order as $h$, which is $\bigo{\zetad}$ for $\X_k$:
\[
    \frac{7}{9}L'\abs{\kmin}r^2 + \frac{3}{2} \abs{\kmin} r \max_{y\in \X_k} \norm{\nabla h(y)} \leq \frac{\oldmu'}{2} = \frac{\mu + \L/\zeta{2D}}{2}.
\] 
Since $r \leq \RR$, $\zeta{2D} \leq 2\zeta{\D}$, $L' = 2\L$ and $\mu\geq 0$, it is enough to have $23Lr\RR \abs{\kmin}\leq \L/(4\zeta{\D})$. Note that in \cref{alg:instance_of_riemacon}, we ensure $r \leq (92\RR\abs{\kmin}\zeta{2r})^{-1}$ which satisfies the previous inequality and also the initial requirement in Proposition 6.1 in \citep{criscitiello2022negative}. 

After this result, we can use Euclidean machinery on $\hat{h}:\exponinv{x_{k-1}}(\X_k) \to \R$, namely \AGD{} \citep{nesterov2005smooth} with a warm start in order to satisfy the condition in Line \ref{line:subroutine} of \cref{alg:accelerated_gconvex}. The algorithm requires projecting onto the feasible set, and we note that in our case it is a Euclidean ball so the operation is very simple. Indeed, let $\hat{\X}_k \defi \exponinv{x_{k-1}}(\X_k)$ and let $\hat{x} \defi \exponinv{x_{k-1}}(x)$, where $x$ is the center of the prox defining $g$ above. By \citep[Proposition 15]{lin2017catalyst} we have that $\hat{x}' \defi \Pi_{\hat{\X}_k}(\Pi_{\hat{\X}_k}(\hat{x}) - \frac{1}{L'} \nabla \hat{h}(\Pi_{\hat{\X}_k}(\hat{x})))$ is a point that satisfies
\begin{equation}\label{eq:euclidean_warm_start}
    \hat{h}(\hat{x}') - \hat{h}(\hat{y}^\ast_h) \leq \frac{L'}{2} \norm{\hat{y}^\ast_h - \Pi_{\hat{\X}_k}(\hat{x})}^2 \leq \frac{L'}{2} \norm{\hat{y}^\ast_h - \hat{x}}^2,
\end{equation}
where $\hat{y}^\ast_h \defi \argmin_{\hat{y} \in \hat{\X}_k} \hat{h}(y)$ is the minimizer of $\hat{h}$, that is, the exact prox. By \citep{nesterov2005smooth}, the convergence rate of \AGD{} with $\hat{x}'$ as initial point is $\bigo{\sqrt{\frac{L'}{\oldmu'}}\log(\frac{\hat{h}(\hat{x}')-h(\hat{y}^\ast_h)}{\hat{\epsilon}})}$, where the accuracy we will require is ${\hat{\epsilon} \defi \Delta_{k'}\norm{\hat{y}^\ast_h - \hat{x}}_{x_{k-1}}^2/(78\lambda)}$, which is less than the $ \Delta_{k'}\dist(\hat{x}, \hat{y}^\ast_h)^2/(78\lambda)$ accuracy required by \cref{alg:accelerated_gconvex}. Here $k'$ is the internal counter for \cref{alg:accelerated_gconvex} and we used the reasoning above yielding that the condition number of $\hat{h}$ is $\bigo{\frac{L'}{\oldmu'}} =\bigo{\zetad}$. Using \eqref{eq:euclidean_warm_start}, we conclude that it is enough to run \AGD{} for $\bigo{\zetad^{\frac{1}{2}}\log(\frac{78\lambda L'}{2\Delta_{k'}})} = \bigotilde{\zetad^{\frac{1}{2}}}$ gradient oracle queries.

\begin{remark}\label{remark:bounded_operator_norm_of_curvature_tensor}
    We can make \cref{alg:instance_of_riemacon} work under a weaker assumption than \cref{assump:bounded_curvature_tensor} after a minor modification on the algorithm. Because the algorithm in \citep{criscitiello2022negative} can work with bounded $\norm{\nabla \curvtensor} \leq F$ for a constant $F$, we can use it as a subroutine in this more generic case. In such a case, the diameter of the balls $\X_k$ must be $\D \leq \frac{\abs{\kmin}}{2F}$, and it is enough to change the condition in Line \ref{line:first_if_in_instanced_alg} to $2\RR \leq \min\{ (46\RR \abs{\kmin} \zeta{2\RR})^{-1}, \abs{\kmin}/(2F)\}$ and if this condition is not satisfied, then after computing $\D$ in Line \ref{line:computing_D_in_instanced_alg}, we further update $\D \gets \min\{\D, \abs{\kmin}/(2F)\}$. In this way, the condition is satisfied and the geometric penalty is $\bigotilde{\frac{\RR}{\D}} = \bigotilde{\zetar^2 + \frac{\RR F}{\abs{\kmin}}}$ instead of $\bigotilde{\zetar^2}$.
\end{remark}

\section[Metric projections, RGD, and metric-projected RGD]{Metric projections, \texorpdfstring{\RGD{}}{RGD}, and metric-projected \texorpdfstring{\RGD{}}{RGD}}\label{sec:RGD_and_proj_RGD}

\begin{definition}[Metric projection operator]
    Let $\M$ be a Riemannian manifold and let $\X \subset \M$ be a closed g-convex subset of $\M$. A \emph{metric projection operator} onto $\X$ is a map $\newtarget{def:projection_operator}{\proj} : \M \to \X$ satisfying $\dist(x, \proj(x)) \leq \dist(x, y)$ for all $y \in \X$.
\end{definition}

A consequence of the definition in Hadamard manifolds is that the projection is single valued and non-expansive, the latter meaning $\dist(\proj(x), \proj(y)) \leq \dist(x, y)$, cf. \citep[Thm 2.1.12]{bacak2014convex}.

The following proposition establishes that computing a metric projection oracle for Riemannian balls has a simple expression, as long as all points involved are in a uniquely geodesic and g-convex set. This is a useful tool for our constrained \RGD{} in this section. More importantly, we also use it for a Riemannian warm start presented in next section, essential for the generic implementation of \hyperref[alg:accelerated_gconvex]{\color{black}Riemacon}'s subroutines.  

\begin{proposition}\label{prop:metric_projection_in_balls}
    Let $\NN\subset\M$ be a uniquely geodesic and g-convex set in a manifold $\M$ and let $\X \defi \expon{x}(\ball(0, r))\subset\NN$ be a Riemannian closed ball of radius $r$ contained in $\NN$. For $y\in\NN $, we have $\proj(y)=y$ if $y\in\X$ and $\proj(y) = \expon{x}(r\exponinv{x}(y)/\norm{\exponinv{x}(y)})$ otherwise. 
\end{proposition}

\begin{proof}
    Due to the uniquely geodesic and g-convexity properties, a curve between two points in $\NN$ is a geodesic in $\NN$ if and only if it is globally distance-minimizing. By definition, a projection point $\proj(y)$ satisfies $\dist(y, \proj(y)) \leq \dist(y, z)$ for all $z\in\X$, so if $y$ is outside of the ball then $\proj(y)$ is on the border of the ball and $\dist(x, \proj(y)) = r$ and so the two geodesic segments from $y$ to $\proj(y)$ and from $\proj(y)$ to $x$ form a curve that is globally distance minimizing between $x$ and $y$ and thus this curve is the geodesic segment joining these points. Our expression for $\proj(y)$ is precisely this point if $y \not\in\X$. Trivially, $\proj(y) = y$ if $y\in\X$.
\end{proof}

We note that the unique geodesic and g-convexity conditions required in \cref{prop:metric_projection_in_balls} always hold for $\NN$ being a ball in a Hadamard manifold, such as $\NN = \expon{x}(\ball(0,\dist(x,y)))$. For other manifolds, for which the maximum sectional curvature is $\kmax > 0$, the condition is satisfied for $\NN = \expon{x}(\ball(0, \hat{R}))$ if $\hat{R} < \min\{\operatorname{inj}(x)/2,  \pi/(2\sqrt{\kmax}) \}$, where $\operatorname{inj}(x)$ is the injectivity radius of $x$ \citep[Thm. IX.6.1]{chavel2006riemannian}. 

As we mentioned before \cref{thm:psi_is_lyapunov}, to the best of our knowledge there is no convergence analysis for metric-projected \RGD{} in the smooth case. From an initial point $x_0$ and for a g-convex set $\X$, this algorithm applies the following update sequentially $x_{t+1} \gets \proj(\expon{x_t}(-\eta \nabla f(x_t)))$. We show linear convergence under smoothness and strong g-convexity for a set of certain constant diameter. Note that for $\X$ being a uniquely geodesic ball, we can use the projection operator in \cref{prop:metric_projection_in_balls} in order to implement this algorithm.

\begin{proposition}
    Let $\M$ be a Riemannian manifold of sectional curvature bounded in $[\kmin, \kmax]$ and let $D > 0$ be such that $\zetad \defi \zeta{D} < 2$. For an initial point $x_0$, and for a g-convex set $\X \defi \expon{x_0}(\ball(0, D/2))\subset\M$, let $f:\M \to \R$ be a differentiable function with a global minimizer $x^\ast$ in $\X$ and let $f$ be smooth and $\mu$-strongly g-convex in $\expon{x^\ast}\ball(0, 2D) \subset \M$. Then, metric-projected \RGD{} with learning rate $\eta = (2-\zetad)/L$ and updates $x_{t+1} = \proj(\expon{x_k}(-\eta\nabla f(x_k)))$ for $k \geq 1$ converges linearly, in the following sense:
    \[
    \dist(x_{t+1}, x^\ast)^2 \leq \left(1- \frac{2\mu(2-\zetad)}{\L}\right) \dist(x_t, x^\ast)^2.
    \] 
\end{proposition}

\begin{proof}
    Let $\tilde{x}_{t+1} \defi \expon{x_t}(-\eta \nabla f(x_t))$ be the unprojected point. By smoothness and $\dist(x_t, x^\ast) \leq D$ since $x_t, x^\ast \in \X$ we have $\norm{\nabla f(x_t)} \leq DL$. Given that $\eta \leq 1/L$, since $\zetad\geq 1$, we have $\dist(\tilde{x}_{t+1}, x^\ast) \leq 2D$. Thus, we can use the smoothness inequality below, where we also use strong g-convexity:
\begin{align}\label{ineq:aux_bubecks_notes}
 \begin{aligned}
    0 &\leq f(\tilde{x}_{t+1}) - f(x^\ast) = f(\tilde{x}_{t+1}) - f(x_t) + f(x_t)  - f(x^\ast) \\
    & \leq \innp{\nabla f(x_t), \tilde{x}_{t+1} - x_t} + \frac{\L}{2} \norm{\tilde{x}_{t+1} - x_t}_{x_t}^2 + \innp{\nabla f(x_t), x_t - x^\ast} - \frac{\mu}{2}\norm{x_t -x^\ast}_{x_t}^2\\
    & = \innp{\nabla f(x_t), \tilde{x}_{t+1} - x^\ast} + \frac{\L\eta^2}{2} \norm{\nabla f(x_t)}^2  - \frac{\mu}{2}\norm{x_t -x^\ast}_{x_t}^2 \\
    & = \innp{\nabla f(x_t), x_t - x^\ast} + (\frac{\L\eta^2}{2} - \eta) \norm{\nabla f(x_t)}^2  - \frac{\mu}{2}\norm{x_t -x^\ast}_{x_t}^2.
 \end{aligned}
\end{align}
Now, we have the following bound, bounding the distance to the minimizer, from which we will derive convergence rates for projected RGD: 
\begin{align}\label{ineq:aux_distance_bound_in_proj_RGD}
 \begin{aligned}
     \dist(\tilde{x}_{t+1}, x^\ast)^2 &\circled{1}[\leq] \zetad \dist(\tilde{x}_{t+1}, x_t)^2 + \dist(x^\ast, x_t)^2 - \innp{\exponinv{x_t}(\tilde{x}_{t+1}), \exponinv{x_t}(x^\ast)} \\
      &\circled{2}[=] (\zetad -1)\eta^2\norm{\nabla f(x_t)}^2 + \norm{x^\ast - \tilde{x}_{t+1}}_{x_t}^2 \\
     &\circled{3}[=] \norm{x^\ast -x_t}_{x_t}^2 + 2\eta \innp{\nabla f (x_t), x^\ast - x_t}_{x_t} + \zetad\eta^2 \norm{\nabla f(x_t)}^2\\
     &\circled{4}[\leq] \left( 2\eta - \frac{\zetad\eta}{1-\frac{\L\eta}{2}}\right) \innp{\nabla f(x_t), x^\ast -x_t}_{x_t} + \left(1 - \frac{\mu\zetad\eta}{1-\frac{\L\eta}{2}} \right) \norm{x^\ast-x_t}_{x_t}^2.
 \end{aligned}
\end{align}
where in $\circled{1}$ we used the Riemannian cosine law \cref{lemma:cosine_law_riemannian} and \cref{remark:tighter_cosine_inequality}, so the geometric constant is  $\zeta{\dist(x_t, x^\ast)} \leq \zetad$. In $\circled{2}$, we used the analogous Euclidean cosine theorem in $T_{x_t}\M$ along with $\dist(\tilde{x}_{t+1}, x_t) = \norm{\eta \nabla f(x_t)}$ which holds by the definition of $\tilde{x}_{t+1}$. Inequality $\circled{3}$  develops the square $\norm{x^\ast - \tilde{x}_{t+1}}_{x_t}^2 = \norm{x^\ast - x_k - \eta \nabla f(x_t)}_{x_t}^2$ and $\circled{4}$ uses \eqref{ineq:aux_bubecks_notes}, where the inequality has been multiplied by $-\zetad\eta^2(\L\eta^2/2-\eta)^{-1} = \frac{\zetad\eta}{1-\frac{\L\eta}{2}}$ which is non negative, since we chose $0 < \eta \leq 1/\L$.

    Now, since by g-convexity we have $\innp{\nabla f(x_t), x^\ast-x_t}_{x_t} \leq 0$, we want to make the factor alongside this term be $\geq 0$ in order to drop it. That means, it should be $2\eta - \frac{\zetad \eta}{1-\frac{\L\eta}{2}} \geq 0$ which is equivalent to $\eta \leq \frac{2-\zetad}{\L}$. By setting $\eta$ exactly to the value $\frac{2-\zetad}{\L}$ and using $\zetad < 2$ by assumption, we have $\frac{\zetad\eta}{1-\frac{\L\eta}{2}} = 2(2-\zetad)/\L$ and so we can conclude:
\[
    \dist(x_{t+1}, x^\ast)^2 \leq \dist(\tilde{x}_{t+1}, x^\ast)^2 \leq \left(1- \frac{2\mu(2-\zetad)}{\L}\right) \dist(x_t, x^\ast)^2.
\] 

\end{proof}

There are different analyses for regular \RGD{}. We include this curvature independent analysis since we could not find it in the literature. Recall that the iterates are defined as:
\begin{equation}\label{eq:unconstrained_RGD}
x_{t+1} \gets \expon{x_t}(x_t - \frac{1}{L} \nabla f(x_t)).
\end{equation}

\begin{theorem}\label{thm:rgd}
    For a Riemannian manifold $\M$, assume $f:\M\to\R$ is a g-convex and differentiable function with a unique minimizer $x^\ast$. Let $\mathcal{M}$ be a Riemannian manifold such that there are geodesics of minimum length between any point and $x^\ast$. Assume $f$ is $\L$-smooth in a uniquely geodesic set $\NN$ and let $\xInit\in\NN$ be an initial point. Assume $x_t$ is well defined and it is in $\NN$. Then, after $T$ time steps the \RGD{} algorithm given by \eqref{eq:unconstrained_RGD} satisfies: 
    \[
        f(x_T) - f(x^\ast) \leq O\left(\frac{\L R^2}{T}\right),
    \] 
    where $R =\max_{t\in\{1, \dots, T\}}\{\dist(x_t, x^\ast)\} \leq \max_{x:f(x)\leq f(x_0)} \norm{\exponinv{x^\ast}(x)}$. 
\end{theorem}

\begin{proof}
    Note that because of convexity and the unique minimizer assumption, $R$ is finite. Let $\Delta_t \defi f(x_t) - f(x^\ast)$. For all $t\geq 0$, we have by g-convexity of $f$ and Cauchy-Schwarz
    \[
        \Delta_t = f(x_{t})-f(x^{\ast}) \leq\innp{\nabla f\left(x_{t}\right), -\exponinv{x_{t}}(x^{\ast})} \leq R \norm{\nabla f(x_{t})}.
    \] 
    Due to smoothness we have $\Delta_{t}-\Delta_{t+1} \geq \frac{1}{2 \L}\left\|\nabla f\left(x_{t}\right)\right\|^{2}$.
    Combining both inequalities one gets
    \[
    \Delta_{t}^{2} \leq 2 \L R^{2}\left(\Delta_{t}-\Delta_{t+1}\right) \Longrightarrow \frac{\Delta_{t}}{\Delta_{t+1}} \leq 2 \L R^{2}\left(\frac{1}{\Delta_{t+1}}-\frac{1}{\Delta_{t}}\right).
    \] 
    Since $\Delta_{t+1} < \Delta_t$, we have $\frac{1}{\Delta_{t+1}} - \frac{1}{\Delta_t} \geq \frac{1}{2\L R^2}$. Therefore,  we obtain by telescoping that at round $T$ we must have ${\frac{1}{\Delta_T} \geq \frac{T}{2 \L R^2}}$, which yields the result.
\end{proof}

\section{Another subroutine, better geometric penalties via another projection oracle}\label{sec:other_subroutine}

This section presents another subroutine for the proximal subproblems in Line \ref{line:subroutine} of \cref{alg:accelerated_gconvex} when $\f$ is $\L$-smooth and for $\lambda = \zeta{2\D}/\L$. This subroutine does not resort to pulling back the function to the tangent space of a point in order to use a Euclidean algorithm. Due to this, we can make the balls in \cref{alg:instance_of_riemacon} be of any size and reduce geometric penalties in our convergence rates, provided that the projection operator in \eqref{eq:other_projection_operator} is implemented. We show how for manifolds satisfying \cref{assump:bounded_curvature_tensor}, this operator results in a convex problem for Riemannian balls of some constant diameter and we show that it yields a $2$-dimensional problem in spaces of constant sectional curvature. 

We proceed to discuss and summarize the implications of the results in this section. For simplicity, we consider $\sqrt{\abs{\kmin}} = \bigo{1}$ in this discussion. We note that one can formally assume this condition without loss of generality by rescaling the manifold and changing $\RR$ accordingly. Indeed, scaling a manifold up makes distances to increase by a factor and the sectional curvature is reduced by the square of that factor, while $\RR\sqrt{\abs{\kmin}}$ remains constant. Also, we note that if $\RR = \bigo{1}$, our geometric penalties are constant, so here we discuss the case in which this condition is not satisfied. In the previous subroutine in \cref{sec:subroutine}, the Lipschitz constant of $\f$ in the balls of \cref{alg:instance_of_riemacon} was estimated to be $\bigo{\L \RR}$ which forced the balls to be of radius $\Theta(1/\RR)$ in order to guarantee the pulled-back function $\hk$ in the proximal subproblem has condition number of the same order as $\hk$. Consequently, our analysis for \cref{alg:instance_of_riemacon} incurs a geometric penalty that is $\bigo{\dist(x_0, x^\ast)/(1/\RR)} = \bigo{\RR^2} =\bigo{\zetar^2}$. The penalty comes from the ratio between $\dist(x_0, x^\ast)$ and the radius of the inexact ball optimization oracle times the constant $\zetad$ introduced by the outer loop of \cref{alg:accelerated_gconvex}, which is $\bigo{1}$ as long as the radius of the balls is $\bigo{1}$. If the subroutine in this section is used to optimize in balls of constant radius, then the geometric penalties of the algorithm are just $\bigo{\RR} = \bigo{\zetar}$. If we optimized in balls of larger radius $\hat{\RR} = \Omega(1)$ or in the extreme case, in only one ball of radius $\RR$ our penalties would still be $\bigo{\RR} = \bigo{\zetar}$ because in such a case, it is $\zetad = \Theta(\hat{\RR})$ so the constants are $\bigo{\RR\zetad/\hat{\RR} } = \bigo{\RR}$. So there is no improvement if the balls are of larger radius, except for one log factor coming from \cref{alg:instance_of_riemacon} that we could remove if we were only using one ball of radius $\RR$. This fact highlights the importance of having a subroutine that works in a ball of constant radius, that is, that obtains linear rates in ball constrained strongly g-convex smooth problems where the radius is $\Theta(1)$. We note that local algorithms, like the one \citep{criscitiello2022negative}, that optimize a function $f$ in a neighborhood whose radius is $O(\sqrt{\mu/\L})$ would not achieve acceleration under the boosted convergence provided by \cref{alg:instance_of_riemacon}. Indeed, because $\sqrt{\mu/\L} \leq 1$, it is $\zetad = \Theta(1)$, but the number of times that we would call the ball optimization oracle is $\bigo{\RR/(\sqrt{\mu/\L})}$, so for an accelerated algorithm that optimizes in the ball with rates $\bigotilde{\sqrt{\L/\mu}}$, up to geometric penalties, we would only be able to guarantee rates of $\bigotilde{\L/\mu}$, up to other geometric penalties. That is, this analysis would only prove unaccelerated convergence.

On the other hand, the subroutine in this section requires the implementation of the projection operator \eqref{eq:other_projection_operator} which is a non-convex problem in general. However, we show that under \cref{assump:bounded_curvature_tensor} or the more general assumption $\norm{\nabla \curvtensor} \leq F$, problem \eqref{eq:other_projection_operator} is convex for a ball of constant radius. Moreover, for the hyperbolic space it is just a $2$-dimensional problem. 

The subroutine in this section assumes access to the operation 
\begin{equation}\label{eq:other_projection_operator}
x_{t+1} = \argmin_{y\in\X}\{\innp{\nabla \f(x_t), y-x_t}_{x_t} + \frac{\L}{2}\dist(x_t, y)^2\} =\expon{x_t}(\argmin_{v \in\exponinv{x_t}(\X)} \norm{-\frac{1}{L}\nabla \f(x_t)- v }_{x_t}^2),
\end{equation}
and the subroutine is defined as the sequential application of this operation, after a warm start, cf. \cref{lemma:warm_start_riemannian_criterion_2}. In the Euclidean case, this subproblem is equivalent to the projection operator applied to the point $\tilde{x}_{t+1} = \expon{x_t}(-\eta \nabla \f(x_t))$. That is, if $\X$ is in the Euclidean space, it is $x_{t+1} = P_{\X}(\tilde{x}_{t+1}) = P_{\X}(x_t-\eta \nabla \f(x_t))$. However in the general Riemannian case, \eqref{eq:other_projection_operator} and the metric-projection operator $P_{\X}(\tilde{x}_{t+1})$ are two different things. 

\begin{proposition}\label{prop:convergence_of_alt_subroutine}
    Let $\M$ be a Riemannian manifold and let $\X \subset\M$ be closed and geodesically convex. Given a $\mu$-strongly g-convex and $\L$-smooth differentiable function $f:\M\to\R$ and $x_0\in\X$, iterating the rule in \eqref{eq:other_projection_operator} yields points that satisfy
    \[
    \f(x_{t+1})-\f(x^\ast) \leq \left(1-\frac{\mu}{2\L}\right)(\f(x_t)-\f(x^\ast)),
    \] 
    where $x^\ast \defi \argmin_{x\in\X}\f(x)$.
        
\end{proposition}
\begin{proof}
We have 
\begin{align*}
 \begin{aligned}
     \f(x_{t+1}) &\circled{1}[\leq] \min_{x\in\X}\left\{\f(x_t) + \innp{\nabla \f(x_t), x - x_t}_{x_t} + \frac{\L}{2}\dist(x, x_t)^2 \right\} \\
     &\circled{2}[\leq] \min_{x\in\X} \left\{\f(x) + \frac{\L}{2}\dist(x, x_t)^2 \right\} \\
     & \circled{3}[\leq] \min_{\alpha \in [0,1]} \left\{\alpha \f(x^\ast) + (1-\alpha) \f(x_t) + \frac{\L\alpha^2}{2}\dist(x^\ast, x_t)^2 \right\} \\
     & \circled{4}[\leq] \min_{\alpha \in [0,1]} \left\{\f(x_t) - \alpha\left(1-\alpha\frac{\L}{\mu}\right)\left(\f(x_t)-\f(x^\ast)\right) \right\} \\
     &\circled{5}[=] \f(x_t) - \frac{\mu}{2\L}(\f(x_t)-\f(x^\ast)).
 \end{aligned}
\end{align*}
    Above, $\circled{1}$ holds by smoothness and \eqref{eq:other_projection_operator}. The g-convexity of $f$ implies $\circled{2}$. 
Inequality $\circled{3}$ results from restricting the minimum to the geodesic segment between $x^\ast$ and $x_t$ so that ${x = \expon{x_t}(\alpha x^\ast + (1-\alpha)x_t)}$. We also use g-convexity of $f$. In $\circled{4}$, we used strong convexity of $f$ to bound $\frac{\mu}{2} \dist(x^\ast, x_t)^2 \leq \f(x_t)-\f(x^\ast)$. Finally, in $\circled{5}$ we substituted $\alpha$ by the value that minimizes the expression, which is $\mu/2\L$. The result follows by subtracting $\f(x^\ast)$ to the inequality above. 
\end{proof}

As we require for the subroutine in Line \ref{line:subroutine} of \cref{alg:accelerated_gconvex}, we have linear convergence, and this subroutine can be used with our algorithm provided that we can find a suitable warm start. Below, we show how this is done and we note that any other Riemannian constrained algorithm with natural linearly convergent rates can be used to solve this step, given that it is initialized with our warm start. The warm start allows to know when to stop the subroutine at the same time that guarantees fast convergence. One should think about this lemma as being applied to $\hk(\cdot) \defi \f(\cdot) + \frac{1}{2\lambda} \dist(\cdot, \xk)^2$. Also, note that in that case we can compute the gradient of $h$ at any point $y\in\X$ as $\nabla h(y) = \nabla \f(y) - \frac{1}{\lambda} \exponinv{y}(\xk)$.

\begin{lemma}[Warm start]\label{lemma:warm_start_riemannian_criterion_2}
    Let $\M$ be a Riemannian manifold, let $x \in \M$, $\X\subset\M$ be a uniquely g-convex set of diameter $\D$ and $h:\M\to\R$ a g-convex and $L'$-smooth function in $\X$. Assume access to a projection operator $\proj$ on $\X$. Let $x' \defi \proj(x)$, $x^{+}\defi \expon{x'}(-\frac{1}{L'}\nabla h(x'))$, $p_0\defi \proj(x^+)$, and $D'\defi \dist(x^{+}, x') =\norm{\nabla h (x')}/L'$. We have that, for all $p\in \X$:
    \[
        h(p_0)-h(p) \leq \frac{\zeta{D'} L'}{2}\dist(x', p)^2 \leq \frac{\zeta{D'} L'}{2}\dist(x, p)^2. 
    \] 
\end{lemma}

\begin{proof}
    
    With the notation of the lemma, we have, by smoothness of $h$, that the following quadratic $Q:\Tansp{x'}\M\to\R$, $v \mapsto h(x') +\frac{L'}{2}\norm{x^{+}-v}_{x'}^2 -\frac{L'}{2}\norm{x^{+}-x'}_{x'}^2$ induces an upper bound on $h$ in $\X$, via $\expon{x'}(\cdot)$. Thus, we have
\begin{align*}
 \begin{aligned}
     -\frac{\zeta{D'} L'}{2}\dist(x, p)^2 + h(p_0) &\circled{1}[\leq] -\frac{\zeta{D'} L'}{2}\dist(x',p)^2 + h(p_0)  \\
     &\circled{2}[\leq] -\frac{\zeta{D'} L'}{2}\dist(x',p)^2 + Q(\exponinv{x'}(p_0)) \\
     &\circled{3}[\leq] -\frac{\zeta{D'} L'}{2}\dist(x',p)^2 + \pa{h(x') + \frac{L'}{2}\dist(x^+, p_0)^2 - \frac{L'}{2}\dist(x^{+},x')^2 }\\
     &\circled{4}[\leq] -\frac{\zeta{D'} L'}{2}\dist(x',p)^2 + \pa{h(x') + \frac{L'}{2}\dist(x^+, p)^2 - \frac{L'}{2}\dist(x^{+},x')^2 }\\
     &\circled{5}[\leq] -L'\innp{\exponinv{x'}(p), \exponinv{x'}(x^{+})} + h(x')\\
     &\circled{6}[=] -L'\innp{\exponinv{x'}(p), -\frac{1}{L'}\nabla h(x')} + h(x')\\
     &\circled{7}[\leq] h(p).
 \end{aligned}
\end{align*}
    We used the projection property of $x' = \proj(x)$ in $\circled{1}$. We used smoothness in $\circled{2}$. In $\circled{3}$, we used the first part of \cref{corol:moving_quadratics:exact_approachment_recession} with $\delta{D'} = 1$, $r=1$, $x\gets x'$, $y\gets x^{+}$, $p\gets p_0$ to bound the estimated distance $\norm{x^+ - p_0}_{x'}$ by the actual distance $\dist(x^+, p_0)$. We used the projection property of $p_0=\proj(x^{+})$ in $\circled{4}$. In $\circled{5}$, we used the version of \cref{lemma:cosine_law_riemannian} in \cref{remark:tighter_cosine_inequality}. We used the definition of $x^{+}$ in $\circled{6}$, and we conclude in $\circled{7}$ by using g-convexity of $h$.
\end{proof}

\begin{remark}
In the proof above, we actually only used that $h$ is star convex and star smooth at $x'$. That is, for all $p \in \X$, we have:
\begin{equation}\label{eq:start_convexity_and_smoothness}
    r(p) \defi h(x') + \innp{\nabla h(x'), p-x'}_{x'}  \leq  h(p) \leq h(x') + \innp{\nabla h(x'), p-x'}_{x'} + \frac{L'}{2}\dist(x, p)^2 \defi s(p).
\end{equation}
    Note that $s(p)$ is essentially the quadratic $Q$, but it is defined in the manifold. A consequence of only requiring \eqref{eq:start_convexity_and_smoothness} is that the applicability of the warm start is even wider. As an example, the warm start lemma can also be applied to the lower bound yielded by g-convexity when computing the gradient of a g-convex function. Similarly to the upper bound yielded by smoothness. Or in other words, one can apply the lemma to the functions $r(p)$ and $s(p)$ defined in \eqref{eq:start_convexity_and_smoothness}, even though these functions are not g-convex or $L'$-smooth.
\end{remark}

\begin{remark}[General linearly convergent subroutine]\label{remark:RGD_can_be_used_as_subroutine}
    Assume we have an unaccelerated algorithm $\mathcal{A}$ that takes a function $h:\M \to \R$ with minimizer at $y^\ast$ when restricted to $\X\subset\M$ and that is $\oldmu'$-strongly g-convex and $L'$-smooth in $\X$, where $\M$ is a Hadamard manifold of bounded sectional curvature and $\X$ is a geodesically-convex compact set with diameter $\D$. Suppose $\mathcal{A}$ returns a point $p_t$ satisfying $h(p_t)-h(y^\ast) \leq \epsilonp$ after querying a gradient oracle for $h$ and a metric-projection oracle $\proj$ for $\X$ for at most $t=\bigo{(\zetad+\frac{L'}{\oldmu'})\log(\frac{(\hk(p_0)-\hk(y^\ast))+L'\dist(p_0, y^\ast)^2}{\epsilonp})}$ times.
    If we apply this algorithm to the objective in the problem of Line \ref{line:subroutine} of \cref{alg:accelerated_gconvex}: $h\gets \hk(y) \defi \f(y) + \frac{1}{2\lambda} \dist(\xk, y)^2$, we have $y^\ast \gets \ykast$, $L' \gets 2\L $ and $\oldmu' \gets \L /\zeta{2\D}$, so the condition number is $L'/\oldmu' = O(\zeta{2\D})=O(\zetad)$. This is computed taking into account that $\f$ is $\L $-smooth and $0$-strongly g-convex and using the $\zeta{2\D}/\lambda$-smoothness and $1/\lambda$-strong g-convexity of the second summand, which is given by \cref{fact:hessian_of_riemannian_squared_distance} and \eqref{eq:x_k_is_not_far_from_X}. If we initialize the method with $p_0 \defi \proj(\expon{\xkp}(-\frac{1}{L'}\nabla \hk(\xkp)))$, where $\newtarget{def:iterate_xp}{\xkp} \defi \proj(\xk)$, then using ($\L /\zeta{2\D}$)-strong g-convexity of $\hk$ to bound $L'\dist(p_0, \ykast)^2 \leq 4\zeta{2\D}(\hk(p_0)-h(\ykast))$, using \cref{lemma:warm_start_riemannian_criterion_2} with $x \gets \xk$, $p\gets \ykast$, and using the guarantees on $\mathcal{A}$, we have that we find a point $\yk$ satisfying $\hk(\yk)-\hk(\ykast) \leq \frac{\Deltak \dist(\xk, \ykast)^2}{78\lambda}$ in $\bigotilde{\zetad}$ queries to the gradient and projection oracles. 
    
   Indeed, let $D'' \defi (L_{\f, \X}+2\L\D/\zeta{2\D})/L'$, where $L_{\f, \X}$ is the Lipschitz constant of $\f$ in $\X$. In \cref{lemma:warm_start_riemannian_criterion_2}, we have 
    \begin{align*}
   \begin{aligned}
       D' &\gets \norm{\nabla \hk(x')}/L' \leq (\norm{\nabla \f(x')} + L\norm{\exponinv{\xk}(x')}/\zeta{2\D})/L' \leq (L_{\f, \X}+2\L\D/\zeta{2\D})/L' = D'',
   \end{aligned}
\end{align*}
The number of queries to the gradient oracle is given by
\begin{align*}
   \begin{aligned}
         &\bigol{\zeta{2\D} \log \frac{ (\hk(p_0)-\hk(\ykast))  + L'\dist(p_0, \ykast)^2}{\Deltak \dist(\xk, \ykast)^2/(78 \zeta{2\D}/\L)}} =\bigol{\zetad \log \frac{ 78\zetad \cdot (1+4\zeta{2\D})(\zeta{D'} L'/2)\dist(\xk, \ykast)^2}{\L\Deltak \dist(\xk, \ykast)^2}} \\
         &\quad \quad \quad = \bigol{\zetad \log \left(\frac{\zetad\cdot\zeta{D'}}{\Deltak}}\right) =\bigol{\zetad \log \left(\frac{\zetad\cdot\zeta{D''}}{\Deltak}}\right).
   \end{aligned}
\end{align*}
    Note that we know that on the one hand we can stop the algorithm after $\bigo{\zetad \log (\frac{\zetad\cdot\zeta{D'}}{\Deltak})}$ iterations which is a value we can compute, including constants, since we can compute $D'$. On the other hand  the worst-case complexity can be expressed as $\bigo{\zetad \log (\frac{\zetad\cdot\zeta{D''}}{\Deltak})}$ but we do not need to have access to $L_{\f,\X}/L'$. Note that if there is a point $\xast\in\X$ such that $\nabla \f(\xast)=0$, then we have by smoothness that $L_{\f,\X} = O(\L\D)$ and therefore $D'' = O(\D)$. Also, for \cref{alg:instance_of_riemacon}, it is always $L_{\f,\X} = \bigo{\L(\RR+\D)}$ and so $D'' = \bigo{\RR + \D}$. Also, recall that $\zeta{D''} \leq D'' + 1$.
\end{remark}

\begin{algorithm}
    \caption{Alternative subroutine with projection oracle \eqref{eq:other_projection_operator}}
    \label{alg:other_subroutine}

\begin{algorithmic}[1] 
    \REQUIRE $\hk(y)  = f(y) + \frac{1}{2\lambda} \dist(\xk, y)^2$ from the problem in Line \ref{line:subroutine} of \cref{alg:accelerated_gconvex}.
    \vspace{0.1cm}
    \State $L' \gets 2L $ \Comment{Smoothness of $\hk$}
    \State $x_k' \gets \proj(\xk)$
    \State $p_0 \gets \proj(\expon{x'}(-\frac{1}{L'}\nabla h(x')))$
    \State $D' \gets \norm{\nabla h(x')}/L'$
    \State $T \gets T$ in \cref{remark:RGD_can_be_used_as_subroutine} $= \bigo{\zetad \log \frac{\zetad\cdot\zeta{D'}}{\Deltak}}$
    \hrule
    \vspace{0.1cm}
    \FOR {$k = 1 \textbf{ to } \TT$}
    \State $p_{k} = \argmin_{p\in\X}\{\innp{\nabla h(p_{k-1}), p-p_{k-1}}_{p_{k-1}} + \frac{L'}{2}\dist(p_{k-1}, p)^2\}$ \Comment{Operator \eqref{eq:other_projection_operator}} 
    \ENDFOR
    \State \textbf{return} $p_{\TT}$.
\end{algorithmic}
\end{algorithm}

\begin{remark}[Implementing the projection oracle in \eqref{eq:other_projection_operator}]
    As we discussed at the beginning of this section, it is enough for our purposes to run the subroutine given by the sequential iteration of \eqref{eq:other_projection_operator} in balls $\X_k$ of constant diameter $\D$. Similarly, we assume without loss of generality that $\kmin = -1$. By \cref{fact:hessian_of_riemannian_squared_distance}, the function $\Phi_{\xk}(x) = \frac{1}{2}\dist(\xk, x)^2$ has condition number bounded by $\zeta{4D}$ in $\X_k'\defi\expon{\xk}(\ball(0, 2D))$. In fact, in this case one can see that the condition number is bounded by $\zeta{2D}$, because that is the maximum distance of any point in $\X_k'$ to the minimizer. For any $x\in\X_k$, by \citep{criscitiello2022negative}, the pullback function $v \mapsto \Phi_{\xk}(\expon{x}(v))$ restricted to $v\in\ball(0, \D)\subseteq \exponinv{x}(\X_k')$ is strongly convex with condition number $\leq 3\zeta{2D}$, if $D \leq 1/(4\sqrt{3\zeta{2D}})$ under \cref{assump:bounded_curvature_tensor}, which is satisfied for all $D \leq 1/5$. Thus, the level set $\{v \in T_{x}\M\  | \ \dist(\xk, \expon{x}(v)) \leq \D/2 \} = \exponinv{x}(\X_k)$ is a (strongly) convex set and the operation \eqref{eq:other_projection_operator} consists of a Euclidean projection on this strongly convex Euclidean set. The general case $\norm{\nabla \curvtensor} \leq F$, for $F>0$ is analogous.

    Also, we note that in spaces of constant curvature this problem is a $2$-dimensional problem as we argue in the following. Due to the symmetry of these spaces, the subspace $\expon{x_t}(P)$ is another Riemannian manifold of constant sectional curvature,  where $P = \operatorname{span}\{\exponinv{x_t}(x_k), \exponinv{x_t}(-\frac{1}{L}\nabla f(x_t))\}$. The symmetry of Riemannian balls, makes the intersection of this $2$-dimensional space and the Riemannian ball be another Riemannian ball and the projection lives in this subspace. Indeed, the projection is unique because it is the projection onto a strongly convex set. If the projection in $T_{x_t}\M$ lived outside of this plane $P$, then its symmetric with respect to the plane would also be a projection point, which is a contradiction. 
    
\end{remark}

\section{Geometric lemmas}

In this section, we state and prove \cref{lemma:moving_hyperplanes:exact_approachment_recession}, which is used in the proof of \cref{thm:g_convex_acceleration} to show that the lower bound given by $\f(\ykast) + \innp{\vtky, x-\ykast}$ that is affine if pulled back to $\Tansp{\ykast}$ can be bounded by another function, that is affine if pulled back to $\Tansp{\xk}$. We also include and prove, with some generalizations, some known Riemannian inequalities that are used in Riemannian optimization methods and that we also use. The second part of the following lemma appeared in \citep{kim2022accelerated}. Similarly with the second part of the corollary that follows.

In this section, unless otherwise specified, $\M$ is an $n$-dimensional Riemannian manifold of bounded sectional curvature.

\begin{lemma}\label{lemma:moving_quadratics:inexact_approachment_recession}
    Let $x, y, p \in \M$ be the vertices of a uniquely geodesic triangle $\Tri$ of diameter $D$, and let $z^x \in \Tansp{x}\M$, $z^y \defi \Gamma{x}{y}(z^x) + \exponinv{y}(x)$, such that $y = \expon{x}(rz^x)$ for some $r\in[0,1)$. If we take vectors $a^y\in \Tansp{y}\M$, $a^x \defi \Gamma{y}{x}(a^y) \in \Tansp{x}\M$, then we have the following, for all $\xiOnly \geq \zeta{D}$:
    
\begin{align*}
 \begin{aligned}
     \norm{&z^y+a^y-\exponinv{y}(p)}_y^2 + (\delta{D}-1)\norm{z^y+a^y}_y^2 \\
     &\geq \norm{z^x+a^x-\exponinv{x}(p)}_x^2 + (\delta{D}-1)\norm{z^x+a^x}_x^2 - \frac{\xiOnly-\delta{D}}{2}\left(\frac{r}{1-r}\right)\norm{a^x}_x^2,
 \end{aligned}
\end{align*}
and
\begin{align*}
 \begin{aligned}
     \norm{&z^y+a^y-\exponinv{y}(p)}_y^2 + (\xiOnly-1)\norm{z^y+a^y}_y^2 \\
     &\leq \norm{z^x+a^x-\exponinv{x}(p)}_x^2 + (\xiOnly-1)\norm{z^x+a^x}_x^2 + \frac{\xiOnly-\delta{D}}{2}\left(\frac{r}{1-r}\right)\norm{a^x}_x^2.
 \end{aligned}
\end{align*}
\end{lemma}

\begin{proof}
    Let $\gamma$ be the unique geodesic in $\Tri$ such that $\gamma(0) = x$ and $\gamma(r) = y$. We have $\gamma'(0) = z^x$. Along $\gamma$, we define the vector field $V(t) = \Gamma{0}{t}(\gamma)(z^x - t\gamma'(0))$. Then, it is $V'(t) = -\gamma'(t)$, and $\norm{V(t)} = \norm{a + (1-t)z^x}$. We will make use of the potential $w:[0,r] \to \R$ defined as $w(t) = \norm{\exponinv{\gamma(t)}(x)-V(t)}^2$.
    We can compute
\begin{align}\label{ineq:aux1_quadratics}
 \begin{aligned}
     \frac{d}{dt}w(t) &= 2\innp{D_t(\exponinv{\gamma(t)}(x)-V(t)), \exponinv{\gamma(t)}(x)-V(t)} \\
        &= 2\innp{D_t\exponinv{\gamma(t)}(x), \exponinv{\gamma(t)}(x)} - 2\innp{D_t\exponinv{\gamma(t)}(x), V(t)} \\
        &\quad\quad - 2\innp{D_t V(t), \exponinv{\gamma(t)}(x)} + 2\innp{D_t V(t), V(t)}\\
        &= - 2\innp{D_t(\exponinv{\gamma(t)}(x), V(t)} + 2\innp{D_t V(t), V(t)}.\\
 \end{aligned}
\end{align}
    Now, we bound the first summand. We use that for the function $\Phi_p(x) = \frac{1}{2}\dist(x, p)^2$ it holds, for every $\xiOnly \geq \zeta{D}$:
    \[
        -\frac{\xiOnly-\delta{D}}{2}\norm{v}^2 \leq \innp{\Hess \Phi_p(\gamma(t))[v] -\frac{\xiOnly+\delta{D}}{2}v, v} \leq \frac{\xiOnly-\delta{D}}{2}\norm{v}^2,
    \] 
    due to \cref{fact:hessian_of_riemannian_squared_distance}. So for $\beta \in \{-1, 1\}$ we obtain the following bound:
    \begin{align*}
     \begin{aligned}
         -2\beta &\innp{D_t\exponinv{\gamma(t)}(x), V(t)} = 2\beta  \innp{\Hess \Phi_p(\gamma(t))[\gamma'(t)], V(t)}\\
         & = 2\beta  \innp{(\ \ \Hess \Phi_p(\gamma(t))-\frac{\xiOnly+\delta{D}}{2}I \ \ )[\gamma'(t)], V(t)} + \beta \innp{(\xiOnly+\delta{D})\gamma'(t), V(t)}\\
         & \leq 2\norm{\Hess \Phi_p(\gamma(t))-\frac{\xiOnly+\delta{D}}{2}I}\cdot\norm{\gamma'(t)}\cdot\norm{V(t)} +\beta \innp{(\xiOnly+\delta{D})\gamma'(t), V(t)}\\
         & \leq 2\frac{\xiOnly-\delta{D}}{2}\norm{\gamma'(t)}\cdot\norm{V(t)} +\beta \innp{(\xiOnly+\delta{D})\gamma'(t), V(t)}\\
         & \circled{1}[=] 2\frac{\xiOnly-\delta{D}}{2}\norm{z^x}\cdot\norm{a + (1-t)z^x} +\beta (\xiOnly+\delta{D})\innp{z^x, a + (1-t)z^x}\\
     \end{aligned}
    \end{align*}
    Gauss lemma is used in the last summand of $\circled{1}$. Now, if $\beta = -1$, we have
    \begin{align}\label{ineq:aux2_quadratics}
     \begin{aligned}
         -2 &\innp{D_t\exponinv{\gamma(t)}(x), V(t)} \geq  -2\frac{\xiOnly-\delta{D}}{2}\norm{z^x}\cdot\norm{a + (1-t)z^x} + (\xiOnly+\delta{D})\innp{z^x, a + (1-t)z^x}\\
         & \circled{1}[\geq] -\frac{\xiOnly-\delta{D}}{2(1-t)} (\norm{(1-t)z^x}^2+\norm{a + (1-t)z^x}^2) +(\xiOnly-\delta{D})\innp{z^x, a + (1-t)z^x} -2\delta{D}\innp{-z^x, a+(1-t)b}\\
         & \geq -\frac{\xiOnly-\delta{D}}{2(1-t)} (\norm{a}^2 +2\innp{a +(1-t)z^x}) +(\xiOnly-\delta{D})\innp{z^x, a} -2\delta{D}\innp{-z^x, a+(1-t)b}\\
         & \geq -\frac{\xiOnly-\delta{D}}{2(1-t)} \norm{a}^2  -2\delta{D}\innp{D_t V(t), V(t)}.\\
     \end{aligned}
    \end{align}
    On the other hand, analogously, if $\beta = 1$, we have
    \begin{align}\label{ineq:aux3_quadratics}
     \begin{aligned}
         -2 &\innp{D_t\exponinv{\gamma(t)}(x), V(t)} \leq  2\frac{\xiOnly-\delta{D}}{2}\norm{z^x}\cdot\norm{a + (1-t)z^x} + (\xiOnly+\delta{D})\innp{z^x, a + (1-t)z^x}\\
         & \circled{1}[\leq] \frac{\xiOnly-\delta{D}}{2(1-t)} (\norm{(1-t)z^x}^2+\norm{a + (1-t)z^x}^2) -(\xiOnly-\delta{D})\innp{z^x, a + (1-t)z^x} -2\xiOnly\innp{-z^x, a+(1-t)b}\\
         & \leq \frac{\xiOnly-\delta{D}}{2(1-t)} (\norm{a}^2 +2\innp{a +(1-t)z^x}) -(\xiOnly-\delta{D})\innp{z^x, a} -2\xiOnly\innp{-z^x, a+(1-t)b}\\
         & \leq \frac{\xiOnly-\delta{D}}{2(1-t)} \norm{a}^2  -2\xiOnly\innp{D_t V(t), V(t)},
     \end{aligned}
    \end{align}
    where $\circled{1}$ is Young's inequality $2cd \leq c^2 + d^2$. Combining \eqref{ineq:aux1_quadratics}, \eqref{ineq:aux2_quadratics}, \eqref{ineq:aux3_quadratics}, we obtain
    \[
        -\frac{\xiOnly-\delta{D}}{2(1-t)} \norm{a}^2 -2(\delta{D}-1)\innp{D_t V(t), V(t)} \leq \frac{d}{dt}w(t) \leq \frac{\xiOnly-\delta{D}}{2(1-t)} \norm{a}^2 -2(\xiOnly-1)\innp{D_t V(t), V(t)}. \\
    \] 
    Integrating between $0$ and $r<1$, it results in
    \begin{align*}
     \begin{aligned}
         \frac{\xiOnly-\delta{D}}{2}&\log(1-r) \norm{a}^2 -(\delta{D}-1)(\norm{V(r)}^2-\norm{V(0)}^2) \leq w(r) -w(0)\\
         &\leq -\frac{\xiOnly-\delta{D}}{2}\log(1-r) \norm{a}^2 -(\xiOnly-1)(\norm{V(r)}^2-\norm{V(0)}^2). \\
     \end{aligned}
    \end{align*}
    Using the bound $-\log(1-r) \leq \frac{r}{1-r}$ for $r\in [0,1)$ and using the values of $w(\cdot)$ and $V(\cdot)$, we obtain the result.
\end{proof}

\begin{corollary}\label{corol:moving_quadratics:exact_approachment_recession}
    Let $x, y, p \in \M$ be the vertices of a uniquely geodesic triangle of diameter $D$, and let $z^x \in \Tansp{x}\M$, $z^y \defi \Gamma{x}{y}(z^x) + \exponinv{y}(x)$, such that $y = \expon{x}(rz^x)$ for some $r\in[0,1)$. Then, the following holds
    \[
        \norm{z^y-\exponinv{y}(p)}^2 + (\delta{D}-1)\norm{z^y}^2 \geq \norm{z^x-\exponinv{x}(p)}^2 + (\delta{D}-1)\norm{z^x}^2,
    \] 
and
    \[
         \norm{z^y-\exponinv{y}(p)}^2 + (\zeta{D}-1)\norm{z^y}^2 \leq \norm{z^x-\exponinv{x}(p)}^2 + (\zeta{D}-1)\norm{z^x}^2.
    \] 
\end{corollary}

\begin{proof}
    Use \cref{lemma:moving_quadratics:inexact_approachment_recession} with $a^y = 0$. Note that this corollary allows $r=1$ as well. We obtain this result, by continuity, by taking a limit when $r \to 1$.
\end{proof}

The following is a lemma that is already known and is used extensively in Riemannian first-order optimization. It turns out it is a special case of \cref{corol:moving_quadratics:exact_approachment_recession}.
\begin{corollary}[Cosine-Law Inequalities]\label{lemma:cosine_law_riemannian}
    For the vertices $x, y, p \in \M$ of a uniquely geodesic triangle of diameter $D$, we have
    \[
        \innp{\exponinv{x}(y), \exponinv{x}(p)} \geq \frac{\delta{D}}{2} \dist(x,y)^2 + \frac{1}{2}\dist(p, x)^2 - \frac{1}{2}\dist(p, y)^2.
    \] 
    and
    \[
        \innp{\exponinv{x}(y), \exponinv{x}(p)} \leq \frac{\zeta{D}}{2} \dist(x,y)^2 + \frac{1}{2}\dist(p, x)^2 - \frac{1}{2}\dist(p, y)^2
    \] 
\end{corollary}
\begin{proof}
    This is \cref{corol:moving_quadratics:exact_approachment_recession} for $r=1$. Indeed, given $y\in\Tri$ we can use \cref{corol:moving_quadratics:exact_approachment_recession} with $z^x = \exponinv{x}(y)$. Note that in such a case we have $\norm{z^x} =\dist(x,y)$ and $z^y = 0$. Using $\norm{\exponinv{y}(p)} = \dist(y, p)$ and 
\begin{align*}
 \begin{aligned}
     \norm{z^x-\exponinv{x}(p)} &= \norm{z^x}^2 - \innp{z^x, \exponinv{x}(p)} + \norm{\exponinv{x}(p)}^2 \\
     &= \dist(x, y)^2 -  2\innp{\exponinv{x}(y), \exponinv{x}(p)} + \dist(p, x)^2,
 \end{aligned}
\end{align*}
    we obtain the result.
\end{proof}
\begin{remark}\label{remark:tighter_cosine_inequality}
    Actually, in Hadamard manifolds, if we substitute the constants $\delta{D}$ and $\zeta{D}$ in the previous \cref{lemma:cosine_law_riemannian} by the tighter constants $\delta{\dist(p,x)}$ and $\zeta{\dist(p,x)}$, the result also holds. See \citep{zhang2016first}.
\end{remark}

We now proceed to prove a lemma that intuitively says that solving the exact proximal point problem can be used to lower bound $\f$. Compare the result of the following lemma with the Euclidean equality $\innp{g, p-y} = \innp{g, p-x} + \norm{g}^2$, for $g = x-y$ and $x, y, p \in \R^n$. 

\begin{lemma}\label{lemma:moving_hyperplanes:exact_approachment_recession}
    Let $x, y, p \in \M$ be the vertices of a uniquely geodesic triangle of diameter $D$. Define the vectors $g\defi \exponinv{y}(x) \in \Tansp{y}\M$ and $g^x = \Gamma{y}{x}(g) = -\exponinv{x}(y)\in \Tansp{x}\M$. Then we have
    \[
        \innp{g, \exponinv{y}(p)} \geq \innp{g^x, \exponinv{x}(p)} + \delta{D} \norm{g}^2,
    \] 
    and 
    \[
        \innp{g, \exponinv{y}(p)} \leq \innp{g^x, \exponinv{x}(p)} + \zeta{D} \norm{g}^2.
    \] 
\end{lemma}

\begin{proof}\linkofproof{lemma:moving_hyperplanes:exact_approachment_recession}
    Using the definition of $g$, we have $\circled{1}$ below, by the first part of \cref{lemma:cosine_law_riemannian}:
    \begin{align*}
     \begin{aligned}
         \innp{g, \exponinv{y}(p)} &\circled{1}[\geq] \frac{\delta{D}}{2}\norm{g}^2 + \frac{\dist(y, p)^2}{2} - \frac{\dist(x, p)^2}{2} \\
         & \circled{2}[\geq] \innp{g^x, \exponinv{x}(p)} +\delta{D}\norm{g^x}^2,
     \end{aligned}
    \end{align*}
    and in $\circled{2}$ we used \cref{lemma:cosine_law_riemannian} again but with a different choice of vertices so we have $\frac{\dist(y, p)^2}{2} \geq \frac{\delta{D}}{2} \norm{g^x}^2 + \frac{\dist(x, p)^2}{2} + \innp{g^x, \exponinv{x}(p)}$.

    The proof of the second part is analogous: using the definition of $g$, we have $\circled{1}$ below, by the second part of \cref{lemma:cosine_law_riemannian}:
    \begin{align*}
     \begin{aligned}
         \innp{g, \exponinv{y}(p)} &\circled{1}[\leq] \frac{\zeta{D}}{2}\norm{g}^2 + \frac{\dist(y, p)^2}{2} - \frac{\dist(x, p)^2}{2} \\
         & \circled{2}[\leq] \innp{g^x, \exponinv{x}(p)} +\zeta{D}\norm{g^x}^2,
     \end{aligned}
    \end{align*}
    and in $\circled{2}$ we used \cref{lemma:cosine_law_riemannian} again but with a different choice of vertices so we have $\frac{\dist(y, p)^2}{2} \leq  \frac{\zeta{D}}{2} \norm{g^x}^2+ \frac{\dist(x, p)^2}{2} + \innp{g^x, \exponinv{x}(p)}$.
\end{proof}

\end{document}